\documentclass[a4paper,11pt]{article}
\usepackage{amsmath, amsthm, amssymb}
\usepackage{verbatim}
\usepackage{graphicx}
\usepackage{enumerate}

\newcommand{\RR}{{\cal{R}}}
\newcommand{\T}{{\cal{T}}}
\renewcommand{\P}{{\cal{P}}}

\newcommand{\eps}{{\varepsilon}}
\newcommand{\pmX}{{X^{\pm 1}}}
\newcommand{\ws}[1]{\mathcal{W}(#1)}
\newcommand{\len}[1]{\left\vert{#1}\right\vert}
{\catcode`\|=\active
  \gdef\Set#1{\left\{\:{\mathcode`\|"8000\let|\SetVert #1}\:\right\}}}
{\catcode`\|=\active
  \gdef\GPres#1{\left\langle\:{\mathcode`\|"8000\let|\SetVert #1}\:\right\rangle}}
\def\SetVert{\egroup\;\middle|\;\bgroup}

\DeclareMathOperator{\Cay}{Cay}

\newtheorem{theorem}{Theorem}
\newtheorem{corollary}[theorem]{Corollary}
\newtheorem{lemma}[theorem]{Lemma}

\newtheorem{proposition}[theorem]{Proposition}

\newtheorem{remark}[theorem]{Remark}
\newtheorem{notation}[theorem]{Notation}
\newtheorem{observation}[theorem]{Observation}

\theoremstyle{definition}
\newtheorem{definition}[theorem]{Definition}

\title{Algebraic $C(4) \& T(4)$ groups are bi-automatic}
\author{Uri Weiss \\ uriw@tx.technion.ac.il}

\begin{document}

\maketitle

\begin{abstract}
S. Gersten and H. Short have proved that if a group has a presentation which satisfies the algebraic $C(4) \& T(4)$ small-cancellation condition then the group is automatic. Their proof contains a gap which we aim to close. To do that we distinguish between algebraic small-cancellation conditions and geometric small cancellation conditions (which are conditions on the van Kampen diagrams). We show that, under certain additional requirements, geometric $C(4) \& T(4)$ small-cancellation conditions imply bi-automaticity. The additional requirements include a restriction on the labels of edges in minimal van Kampen diagrams. This, together with the so-called barycentric sub-division method proves the theorem.
\end{abstract}

\section{Introduction}

A group presentation $\P = \GPres{X|\RR}$ satisfies the \emph{algebraic} $C(4) \& T(4)$ small-cancellation condition if the possible cancellations between the different relators are limited. To be more rigorous we must first define the notion of a piece. A piece in the presentation $\P$ is a word $W$ over the alphabet $X\cup X^{-1}$ such that there are two different cyclic conjugates of relators such that $W$ is a prefix of both of them (see Section \ref{sec:perliminaries} for exact definitions of these notations). We say that the presentation $\P$ satisfies the algebraic $C(4)$ and the $T(4)$ conditions if:
\begin{description}
	\item \textbf{$C(4)$ condition}: If $R$ is a cyclic conjugate of a relator in $\RR$ and $R = P_1 P_2 \cdots P_k$ is a decomposition of $R$ into a product of $k$ pieces (i.e., $P_i$ is a piece for $i=1,2,\ldots,k$) then $k\geq 4$.
	\item \textbf{$T(4)$ condition}: If $R_1$, $R_2$, and $R_3$ are any three cyclic conjugates of relators in $\RR$ then at least one of the following products is freely reduced as written: $R_1R_2$, $R_2R_3$, or $R_3R_1$.
\end{description}
A group satisfies the algebraic $C(4) \& T(4)$ small-cancellation condition if it has a presentation which satisfies these conditions.

In the paper \cite{GS91} of S.\ Gersten and H.\ Short it is proven that all groups which satisfy the algebraic $C(4) \& T(4)$ conditions are automatic. The proof outlined in the paper contains a gap that we intend to close in this work; see the Appendix for a detailed explanation of the problems in the proof given in \cite{GS91}. The main contribution of the current work is a new proof of the following theorem:

\begin{theorem} \label{thm:mainApplicationC4T4}
Groups satisfying the algebraic $C(4) \& T(4)$ condition are bi-automatic.
\end{theorem}

It is a well known fact that a minimal van Kampen diagram (see Section \ref{sec:vanKampen}) over a presentation that satisfies the algebraic $C(4) \& T(4)$ conditions has the property that every inner region has at least four neighbors and there are no inner vertices of valance three. We call such diagrams \emph{$C(4) \& T(4)$ diagrams}. We say that a presentation satisfies the \emph{geometric} $C(4) \& T(4)$ small-cancellation condition if all \emph{minimal} van Kampen diagrams over this presentation are $C(4) \& T(4)$ diagrams. Thus, if a presentation satisfies the algebraic condition it follows that it satisfies the geometric condition (the other direction is not true). Notice that we used ``$C(4) \& T(4)$'' in three different objects (two types of presentations and one type of diagrams). 

The proof of Theorem \ref{thm:mainApplicationC4T4} in \cite{GS91} starts by applying the so-called ``barycentric sub-division'' (which we describe in Section \ref{sec:barySubDiv}). This is also the first step we do. If this procedure is applied to a presentation which satisfies the algebraic $C(4) \& T(4)$ conditions then the result is a presentation which satisfies the \emph{geometric} $C(4) \& T(4)$ small-cancellation condition (and other important properties) but \emph{not} the algebraic conditions. Thus, the technique which is used for algebraic conditions must be generalized to handle the weaker geometric conditions.

Let $G$ be a group having a presentation $\P$ which satisfies the algebraic $C(4) \& T(4)$ condition. S. Gersten and H. Short have shown in \cite{GS90} that if in addition all relators of $\P$ are of length four and all pieces of $\P$ are of length one then the group $G$ is automatic. The additional assumptions on the length of pieces imply that minimal van Kampen diagrams over the presentation $\P$ have the property that a path on the boundary of two regions is labelled by a generator. To complete the proof of Theorem \ref{thm:mainApplicationC4T4} we extend the result from \cite{GS90} to the geometric setup:

\begin{theorem} \label{thm:mainThmC4T4}
Let $G$ be a group and assume that it has a presentation $\P$ such that the following conditions hold:
\begin{enumerate}
	\item All relators of $\P$ are of length four and cyclically reduced.
	\item If $M$ is a \emph{minimal} van Kampen diagram over the presentation $\P$ then:
	\begin{enumerate}
		\item $M$ is a $C(4) \& T(4)$ diagram.
		\item If $\rho$ is a path of $M$ which is on the boundary of two regions then $\rho$ is labelled by a generator.
	\end{enumerate}
\end{enumerate}
Then, $G$ is a bi-automatic group.
\end{theorem}

The main difference between the algebraic conditions and the geometric conditions is that in the geometric setup there may be diagrams which contain vertices of valence three (under the algebraic setup diagrams may never contain vertices of valence three). However, vertices of valence three appear only in \emph{non-minimal diagrams}.

The rest of the paper is organized as follows. In Section \ref{sec:perliminaries} we give the basic notations and definitions. In Section \ref{sec:vanKampen} we give the necessary background on van Kampen diagrams and of $C(4) \& T(4)$ diagrams. In Section \ref{sec:proof} we prove Theorem \ref{thm:mainThmC4T4}. In Section \ref{sec:barySubDiv} we describe the so-called ``barycentric sub-division'' procedure and its properties when applied to an algebraic $C(4) \& T(4)$ presentation. Finally, in the small Section \ref{sec:application} we give the proof of Theorem \ref{thm:mainApplicationC4T4}. An appendix is included which analyze the problems in the original proof of Gersten and Short in \cite{GS91}.

This work is part of the author's Ph.D. research conducted under the supervision of Professor Arye Juh\a'{a}sz.

\section{Perliminaries \label{sec:perliminaries}}

\begin{notation}\label{not:standardNot}
Let $\GPres{X|\RR}$ be a finite presentation for a group $G$ (we will always assume that the elements of $\RR$ are cyclically-reduced). Denote by $\ws{X}$ the set of all finite words with letters in $\pmX=X\cup X^{-1}$. We denote by $\eps$ the empty word. The elements of $\ws{X}$ are not necessarily freely-reduced. Let $W$ and $U$ be words in $\ws{X}$. We use the following notations:
\begin{enumerate}
 \item $\overline{W}$ denotes the element in $G$ which $W$ presents. The projection map $\pi:\ws{X}\to G$ which sends $W$ to $\overline{W}$ is called the \emph{natural map}. We will say that $W=U$ in $G$ if $\overline{W} = \overline{U}$.

 \item $|W|$ is the length of $W$ (i.e., the number of letters in $W$).

 \item $W$ is called \emph{geodesic} in $G$ if for every $U\in\ws{X}$ such that $\overline{W}=\overline{U}$ we have $|W|\leq|U|$.

 \item $W(n)$ is the prefix of $W$ consisting of the first $n$ letters of $W$. If $n>|W|$ then $W(n)=W$.

 \item The \emph{symmetric closure} of $\RR$ is the finite subset of $\ws{X}$ which consists of all cyclic conjugates of elements of $\RR$ and their inverses. If $\RR$ is equal to its symmetric closure then we say that $\RR$ is \emph{symmetrically closed}.

 \item $P\in\ws{X}$ is an \emph{$\RR$-piece} (or simply a \emph{piece} if $\RR$ is fixed in the context) if there are $W_1,W_2\in\ws{X}$ such that $W_1\neq W_2$ and $PW_1,PW_2$ are relators in the symmetric closure of $\RR$ (and $PW_1,PW_2$ are reduced as written).

\end{enumerate}
\end{notation}

\begin{definition}[Cayley Graph and the associated metric]
The Cayley graph of a group $G$ with generating set $X$ is the graph whose vertex set is $G$ and there is a directed edge from $g_1$ to $g_2$ labelled by $x\in\pmX$ if and only if $g_1x=g_2$. We denote this graph by $\Cay(G,X)$. The metric $d(\cdot,\cdot)=d_X(\cdot,\cdot)$ in $\Cay(G,X)$ is the standard non-directed \emph{path metric} (also called the \emph{word-metric} of $G$). Namely, $d(g_1,g_2)$ is the edge length of the shortest path from $g_1$ to $g_2$. Each word $W=x_1x_2\cdots x_k$ in $\ws{X}$ corresponds naturally to a path in $\Cay(G,X)$ whose vertices are $1, \overline{x_1}, \overline{x_1x_2}, \overline{x_1x_2x_3}, \ldots, \overline{W}$. For two words $W,U\in\ws{X}$ we denote by $d_X(W,U)$ the distance between $\overline{W}$ to $\overline{U}$ in $\Cay(G,X)$.
\end{definition}

\begin{definition}[Fellow Travellers \cite{EPS92}] \label{def:FellowTraveller}
Let $k$ be a positive number. Two words $W,U\in \ws{X}$ are called \emph{$k$-fellow-travelers} if for all $\ell\in\mathbb{N}$ we have:
\[
d_X\left(W(\ell),U(\ell)\right)\leq k
\]
Suppose $L\subseteq\ws{X}$. $L$ has the \emph{fellow-travelers property} if there is a constant $k$ for which the following condition holds: if $W$ and $U$ are elements of $L$ and $x,y\in\pmX\cup\Set{\varepsilon}$ such that $xW=Uy$ in $G$ then $xW$ and $Uy$ are $k$-fellow-travelers.
\end{definition}

Note that if $W=UV$ in $G$ then $d(W,U) \leq |V|$. Next, the definition of bi-automatic groups.

\begin{definition}[Bi-automatic structure and Bi-automatic group \cite{EPS92}] \label{def:biAuto}
A \emph{bi-automatic structure} of $G$ with generating set $X$ is a regular \cite{HU79} subset of $\ws{X}$ which maps onto $G$ under the natural map and which has the fellow-traveler property. A group is \emph{bi-automatic} if it has a bi-automatic structure.
\end{definition}

Bi-automaticity is a property which does not depend on the generating set \cite[Thm. 2.4.1]{EPS92}. As evident from the definition above, fellow-traveling property plays an important role when one tries to establish bi-automaticity of a groups. The following observation will be useful for checking that property.

\begin{observation} \label{obs:CheckingFeloTrvlProp}
Let $G$ be a group finitely generated by $X$ and let $W$ and $W'$ be two elements of $\ws{X}$. Suppose $W$ and $W'$ decompose as $W=V_1 U V_2$ and $W' = V_1 U' V_2$. In this case, $W$ and $W'$ are $(|U|+|U'|)$-fellow-travelers.
\end{observation}

We next describe a technique known as `falsification by fellow-traveller' due to Davis and Shapiro \cite{DS91}. This technique was used several times to prove bi-automaticity of different groups (see, for example, \cite{Pei96, Wei07}). We start with a technical definition of a function that combine two words into a single word. This will allows us to define the idea of regular orders.

\begin{definition} \label{def:deltaX}
Let $\Sigma$ be a set not containing $\$$. Denote by $\Sigma(2,\$)$ the set $\Sigma\cup\Set{\$}\times \Sigma\cup\Set{\$} \setminus \Set{(\$,\$)}$. The map $\delta_\Sigma$ is the map $\delta_\Sigma:\Sigma^*\times \Sigma^* \to \Sigma(2,\$)^*$ which is defined as follows. Let $(W,U) \in \Sigma^*\times \Sigma^*$ where $W=x_1\,x_2\,\cdots\,x_n$ and $U=y_1\,y_2\,\cdots y_m$. Then,
\[
\delta_\Sigma(W,U) = \left\{
\begin{array}{ll}
  (x_1,y_1)\cdots(x_n,y_n)(\$,y_{n+1})\ldots(\$,y_{m}) & n<m \\
  (x_1,y_1)\cdots(x_m,y_m)(x_{m+1},\$)\ldots(x_{n},\$) & m<n \\
  (x_1,y_1)\cdots(x_n,y_m) & m=n
\end{array}
\right.
\]
\end{definition}

\begin{definition} \label{def:regOrder}
A regular order on a set of words $\Sigma^*$ is an order $S \subseteq \Sigma^*\times \Sigma^*$ such that $\delta_\Sigma(S)$ is a regular language.
\end{definition}

Next, the `falsification by fellow-traveller' technique.

\begin{proposition} [Falsification by fellow-traveller; Lemma 29 of \cite{Pei96}] \label{prop:falseKFT}
Let $G$ be finitely generated by $X$ and let ``$\prec$" be a regular order on $\ws{X}$ that has minimal element for every non-empty subset of $\ws{X}$. Assume there is a positive constant $k$, such that for every word $W$ that is not ``$\prec$"-minimal there is a word $V$ with the following properties:
\begin{enumerate}
 \item $V\prec W$.
 \item $V=W$ in $G$.
 \item $V$ and $W$ are $k$-fellow-travelers.
\end{enumerate}
Then, the set of ``$\prec$"-minimal words is a regular set and maps onto $G$ through the natural map.
\end{proposition}

Proposition \ref{prop:falseKFT} will be used to prove regularity of a bi-automatic structure (this is required in Definition \ref{def:biAuto}). Lemma \ref{lem:husdDist} below is used to establish the fellow-travelling property.
 
\begin{lemma} [Lemma 3.2.3 of \cite{EPS92}] \label{lem:husdDist}
Let $G$ be finitely presented by $\GPres{X|\RR}$, let $W,U\in\ws{X}$ be two geodesics, and let $x,y\in\pmX\cup\Set{\eps}$. Assume that $xW = Uy$ in $G$ and that there is a positive constant $s\in\mathbb{N}$ such that for every prefix $P_1$ of $xW$ there is a prefix $P_2$ of $Uy$ and $V\in\ws{X}$ with $\len{V}\leq s$ for which $P_1 V=P_2$ in $G$. Assume further that the same holds when the roles of $xW$ and $Uy$ are exchanged. Then, $xW$ and $Uy$ are $(2s+1)$-fellow-travelers.
\end{lemma}

\begin{remark}
In the terminology of \cite{EPS92}, the words $xW$ and $Uy$ in Lemma \ref{lem:husdDist} are of $s$-Hausdorff distance (see \cite{EPS92}).
\end{remark}

\section{van Kampen diagrams \label{sec:vanKampen}}

One of the basic tools of small cancellation theory is van Kampen diagrams. We next give the usual definitions and notations taken mainly from \cite[Chapter V]{LS77}. Lengths of paths in a van Kampen diagram may be used to estimate distances in the Cayley graph. This is useful when trying to establish the fellow-traveler property.

\begin{definition}
A \emph{map} $M$ is a finite planar connected 2-complex (see \cite[Chapter V]{LS77}). We use the common convention and refer to the $0$-cells, $1$-cells, and $2$-cells of $M$ as \emph{vertices}, \emph{edges}, and \emph{regions}, respectively.
\end{definition}

All maps are assumed to be connected and simply connected unless we note otherwise. Vertices of valence one or two are allowed. Regions are open subset of the plane which are homeomorphic to open disk and edges are images of open interval. A \emph{sub-map} of $M$ is a map whose vertices, edges, and regions are also regions, edges, and regions of $M$, respectively. Each edge $e$ of $M$ is equipped with an orientation (i.e., a specific choice of beginning and end) and we denote by $e^{-1}$ the same edge but with reversed orientation; $i(e)$ will denote the beginning vertex of $e$ and $t(e)$ will denote the ending vertex of $e$. A \emph{path} in a map $M$ is a sequence of edges $e_1,e_2,\ldots,e_k$ such that $t(e_i)=i(e_{i+1})$ for all $i=1,2,\ldots,k-1$. Given a path $\delta=e_1 e_2 \cdots e_k$ we define $i(\delta)$ to be $i(e_1)$ and $t(\delta)$ to be $t(e_k)$; we denote by $\delta^{-1}$ the reversed path, namely, $e_k^{-1} e_{k-1}^{-1} \cdots e_1^{-1}$. For a path $\delta$ we denote by $|\delta|$ the length of $\delta$ which is the number of edges it contains. The case where a path $\rho$ has length zero is allowed and in this case $\rho$ is a single vertex (which is the initial and terminal vertex of $\rho$). If $\mu$ is the path $e_1 \cdots e_r$ and $\rho$ is the path $e_{r+1} \cdots e_s$ such that $t(\mu) = i(\rho)$ then the concatenation $\mu\rho$ of $\mu$ and $\rho$ is defined and is the path $e_1\cdots e_s$. If $\mu = \mu_1 \mu_2 \mu_3$ then $\mu_1$ is a prefix of $\mu$, $\mu_2$ a sub-path of $\mu$, and $\mu_3$ a suffix of $\mu$. A \emph{spike} is a vertex of valence one in $M$. A \emph{boundary path} of a map $M$ is a path that is contained in $\partial M$; a \emph{boundary cycle} is a closed simple boundary path. The term \emph{neighbors}, when referred to two regions, means that the intersection of the regions' boundaries contains an edge; specifically, if the intersection contains only vertices, or is empty, then the two regions are not neighbors. \emph{Boundary regions} are regions with outer boundary, i.e., the intersection of their boundary and the map's boundary contains at least one edge. Regions which are not boundary regions are called \emph{inner regions}. \emph{Boundary edges} and \emph{boundary vertices} are edges and vertices on the boundary of the map. \emph{Inner edges} and \emph{inner vertices} are edges and vertices not on the boundary of the map.

We next turn to van Kampen diagrams. These are maps with a specific choice of labeling on the their edges.

\begin{definition} \label{def:vanKampenDiagram}
Let $M$ be a map. A \emph{labeling function} on $M$ with labels in group $F$ is a function $\Phi$ defined on the set of edges of $M$ and which sends each edge to a non-identity element of $F$ such that $\Phi(e^{-1})=\Phi(e)^{-1}$. We naturally extends $\Phi$ to paths in the $1$-skeleton of $M$ by sending a path $e_1 e_2 \cdots e_k$ to $\Phi(e_1) \Phi(e_2) \cdots \Phi(e_k)$. Given a finite presentation $\P=\GPres{X|\RR}$, an \emph{$\RR$-diagram} is a map $M$ together with a labeling function $\Phi$ such that $\Phi(e)\in\ws{X}$ and the images of boundary cycles of regions are elements of the symmetric closure of $\RR$; the map $M$ is the underlying map of the diagram. $\RR$-diagrams will be also referred to as diagrams over the presentation $\P$, \emph{van Kampen diagram}, and sometimes as just diagram if the set $\RR$ or the presentation $\P$ are known.
\end{definition}

We say that two diagrams $M_1$ and $M_2$ are \emph{equivalent} if there is an isomorphism between the labelled 1-skeletons of $M_1$ and $M_2$ which can be extended to the regions of $M_1$ and $M_2$. Essentially, two diagrams are equivalent is they are equal up to homeomorphism of complexes.

Suppose we are given a group $G$ with presentation $\GPres{X|\RR}$. van Kampen theorem \cite[Chap. V]{LS77} states that a word $W\in\ws{X}$ presents the identity of $G$ if and only if there is an $\RR$-diagram with a boundary cycle labelled by $W$. A van Kampen diagram with boundary label $W$ is called \emph{minimal} if it has the minimal number of regions out of all the $\RR$-diagrams with boundary cycle labelled by $W$. Note that a diagram is minimal then every sub-diagram of $M$ is minimal. A van Kampen diagram is \emph{reduced} if for every two neighboring regions $D_1$ and $D_2$ such that $\partial D_1 = \mu \rho$ and $\partial D_2 = \rho^{-1}\sigma$  we have that the label of $\mu\sigma$ cannot be freely reduced to $1$. A diagram which is not reduced is not minimal since we can remove regions from it without changing the boundary label (i.e., it has a non-minimal sub-diagram). Thus, a minimal diagram is reduced (but, reduced diagrams may be non-minimal).

A map $M$ with boundary cycle $\delta\mu^{-1}$ is \emph{$(\delta,\mu)$-thin} if every region $D$ has at most two neighbors and its boundary $\partial D$ intersects both $\delta$ and $\mu$. A map is \emph{thin} if it is $(\delta,\mu)$-thin for some boundary paths $\delta$ and $\mu$ and a diagram is thin if its underling map is thin. See Figure \ref{fig:thinEquaDiag} for an illustration of a thin map. If two elements $W$ and $U$ present the same element in a group $G$ then $WU^{-1}=1$ in $G$ and so by the van Kampen theorem there is a van Kampen diagram $M$ with boundary label $WU^{-1}$. We call such diagram \emph{equality diagram for $W$ and $U$}. If in addition the diagram is $(\delta,\mu)$-thin and the labels of $\delta$ and $\mu$ are $W$ and $U$, respectively, then we say that $M$ is a \emph{thin equality diagram for $W$ and $U$}.

\begin{figure}[ht]
\centering
\includegraphics[totalheight=0.17\textheight]{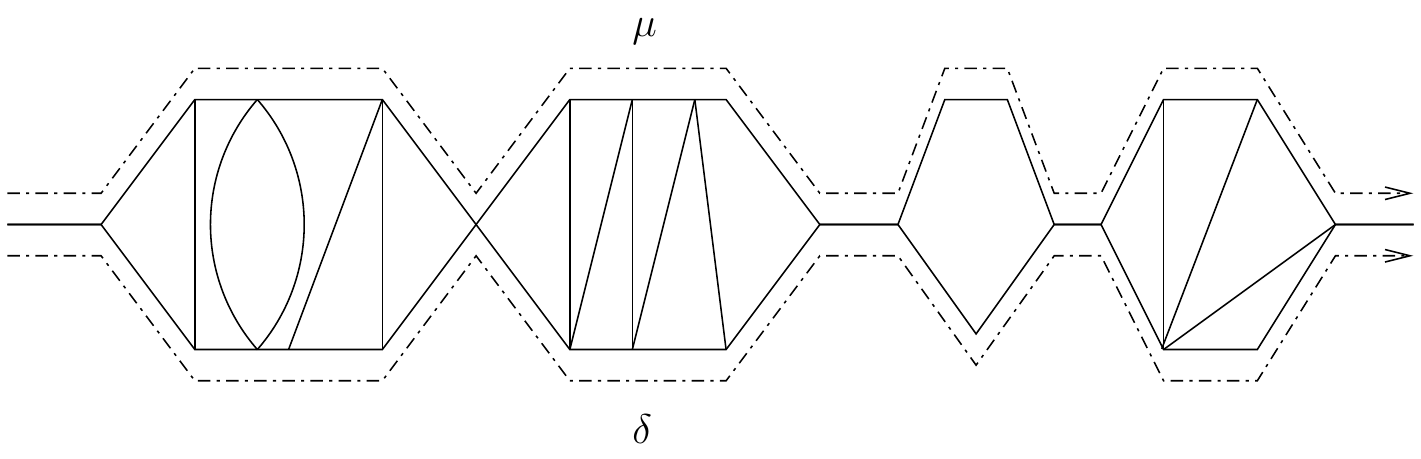}
\caption{An illustration of thin map}\label{fig:thinEquaDiag}
\end{figure}

We next give two lemmas which connect the notion of thin diagrams to the results we need in this work. The first is Lemma 21 from \cite{Pei96}.

\begin{lemma} \label{lem:thinDiagForKFT}
Let $G$ be a group which is finitely presented by $\GPres{X|\RR}$ and let $k$ be the length of the longest relator in $\RR$. Suppose $W$ and $U$ are elements of $\ws{X}$ such that $W=U$ in $G$. Suppose further that $M$ is a $(\delta,\mu)$-thin diagram where $\delta$ is labelled by $W$ and $\mu$ is labelled by $U$. Then, then one of the following holds:
\begin{enumerate}
 \item $W$ and $U$ are $k$-fellow-travellers.
 \item There is a word $W'$ such that $W'$ and $W$ are $k$-fellow-travellers, $|W'|<|W|$ and $W'=W$ in $G$.
 \item There is a word $U'$ such that $U'$ and $U$ are $k$-fellow-travellers, $|U'|<|U|$ and $U'=U$ in $G$.
\end{enumerate}
\end{lemma}

The second lemma show when the conditions of Lemma \ref{lem:husdDist} are satisfied.

\begin{lemma} \label{lem:thinDiagToHaudDist}
Let $G$ be a group which is finitely presented by $\GPres{X|\RR}$. Suppose that $W$ and $U$ are geodesic elements of $\ws{X}$ and $x,y \in \pmX \cup \Set{\eps}$ such that $xW = Uy$ in $G$. Assume that $M$ is a $(\delta,\mu)$-thin diagram where $\delta$ is labelled by $xW$ and $\mu$ is labelled by $Uy$ and let $s$ be the maximal length of a relator in $\RR$. In this case the conditions of Lemma \ref{lem:husdDist} hold for the given $s$.
\end{lemma}
\begin{proof}
Let $P_1$ be a prefix of $xW$. Then, there is a decomposition $\delta = \delta_1 \delta_2$ where $\delta_1$ is labelled by $P_1$. Let $\rho$ be the shortest path in $M$ that connects the terminal vertex $t(\delta_1)$ to a vertex $v$ of $\mu$. Since $M$ is thin then clearly $|\rho|\leq s$. Decompose $\mu = \mu_1 \mu_2$ where the terminal vertex of $\mu_1$ is $v$. Let $P_2$ be the label of $\mu_1$. Then, $V_2$ is a prefix of $Uy$. Let $V$ be the label of $\rho$. We have that since $\delta_1 \rho \mu_1^{-1}$ is a closed loop in $M$ that $P_1 V P_2^{-1} = 1$ in $G$ and so $P_1 V  = P_2$ in $G$ as needed.
\end{proof}

The rest of this section is devoted to the definition of $C(4)\& T(4)$ maps and their properties. We start with the definition.

\begin{definition}[$C(4)\& T(4)$ maps and diagrams] \label{def:C4T4Diag}
Let $M$ be a map. We say that $M$ is a $C(4)\& T(4)$ map if the following two conditions hold:
\begin{enumerate}[(a)]
	\item\label{def:C4T4Diag:a} If $D$ is an inner region then $D$ has at least four neighbors.
	\item\label{def:C4T4Diag:b} If $v$ is an inner vertex then $v$ does not have valence three (however, valence two is allowed).
\end{enumerate}
A diagram is a $C(4)\& T(4)$ diagram if its underlying map is a $C(4)\& T(4)$ map.
\end{definition}

\begin{remark}
If $\P$ is an algebraic $C(4)\& T(4)$ presentation then a reduced van Kampen diagram $M$ over $\P$ is a $C(4)\& T(4)$ diagram. Condition (\ref{def:C4T4Diag:b}) of Definition \ref{def:C4T4Diag} above hold for every diagram over $\P$ (regardless of it being reduced). Condition (\ref{def:C4T4Diag:a}) of Definition \ref{def:C4T4Diag} hold for the following reason. If $M$ is a reduced diagram then if $D_1$ and $D_2$ are two neighboring regions in $M$ and $\alpha$ is a connected path in $\partial D_1 \cap \partial D_2$ then the label of $\alpha$ is a piece. Consequently, using the fact that a boundary label of every region in $M$ cannot be decomposed into a product of less than four pieces, we get that every inner region must have at least four neighbors (we are using here a result from \cite{Juh89} stating that it is enough to check Condition (\ref{def:C4T4Diag:a}) of Definition \ref{def:C4T4Diag} for regions with simple boundary cycle).
\end{remark}

One of the main contribution of this paper is the development of a technique which works with presentation where the second condition hold only in minimal diagrams (but may not hold for all diagrams). Next, we refine the definition of $C(4)\& T(4)$ maps as follows.

\begin{definition}[Proper $C(4)\& T(4)$ maps and diagrams] \label{def:PropC4T4Diag}
Let $M$ be a map. We say that $M$ is a proper $C(4)\& T(4)$ map if it is a $C(4)\& T(4)$ map and the following two conditions hold:
\begin{enumerate}[(a)]
	\item\label{def:PropC4T4Diag:a} There are no inner vertices of valance two.
	\item If $D$ is a region (not necessarily an inner region) then $\partial D$ contains at least four edges.
\end{enumerate}
A diagram is a proper $C(4)\& T(4)$ diagram if its underlying map is a proper $C(4)\& T(4)$ map.
\end{definition}

Condition (\ref{def:PropC4T4Diag:a}) of Definition \ref{def:PropC4T4Diag} is a technical condition which we can guarantee by removing all inner vertices of valence two. However, once Condition (\ref{def:PropC4T4Diag:a}) is satisfied one can count neighbors of an inner region $D$ by simply counting the number of edges in $\partial D$. Consequently, if there are no inner vertices of valence two then the difference between $C(4)\& T(4)$ maps which are not necessarily proper and proper $C(4)\& T(4)$ maps is that in proper $C(4)\& T(4)$ maps we also assume that boundary regions have at least four edges.

We complete the section with a characterization of thinness in proper $C(4)\& T(4)$ maps. The characterization is done through the idea of ``thick configurations''.

\begin{definition}[Thick configurations] \label{def:thickConf}
Let $M$ be a proper $C(4) \& T(4)$ map and let $\alpha$ be a path on the boundary of $M$. A \emph{thick configuration} in $\alpha$ is a sub-diagram $N$ of $M$ where one of the following holds:
\begin{enumerate}
	\item \textbf{Thick configuration of the first type}. $N$ contains single region $D$ with $\partial D = \mu \sigma^{-1}$ such that $\mu = \partial D \cap \alpha$ and $|\mu| > |\sigma|$. See Figure \ref{fig:thickConf}(a).
	\item \textbf{Thick configuration of the second type}. $N$ has connected interior and consists of two neighboring regions $D_1$ and $D_2$. The boundary of $D_2$ decomposes as $\partial D_2 = \mu\sigma^{-1}$ where $|\mu|=|\sigma|=2$, $\mu$ is a sub-path of $\alpha$, and $\sigma$ contains only inner edges. The boundary of $D_1$ contains an outer edge $e$ such that $e \mu$ is a sub-path of $\alpha$. See Figure \ref{fig:thickConf}(b).
\end{enumerate}
If there is a thick configuration along $\alpha$ then we say that \emph{$\alpha$ contains a thick configuration}.

\begin{figure}[ht]
\centering
\includegraphics[totalheight=0.20\textheight]{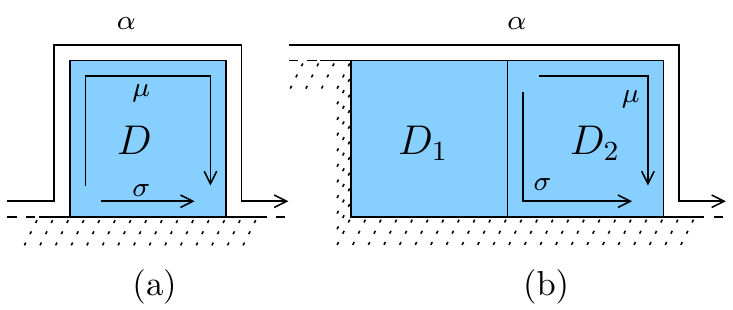}
\caption{Thick Configurations}\label{fig:thickConf}
\end{figure}

\end{definition}

The following theorem characterizes when a proper $C(4) \& T(4)$ diagram is thin. This is a special case of Theorem 13 in \cite{Wei07} (we do not give here the full statement of the theorem since it requires additional definitions which are beyond the scope of this work).

\begin{theorem} \label{thm:noThickToThinDiag}
Let $M$ be a proper $C(4) \& T(4)$ diagram with boundary cycle $\sigma \alpha \tau^{-1} \beta^{-1}$ such that $|\sigma|\leq 1$ and $|\tau|\leq1$. If $\alpha$ and $\beta$ do not contain thick configurations then $M$ is $(\sigma\alpha,\beta\tau)$-thin.
\end{theorem}

\section{From geometric conditions to bi-automaticity \label{sec:proof}}

In this section we prove Theorem \ref{thm:mainThmC4T4}. Let $G$ be a group and let $\P = \GPres{X|\RR}$ be a presentation of $G$ for which the properties of Theorem \ref{thm:mainThmC4T4} hold. The group $G$ and the presentation $\P$ are fixed throughout this section. Recall that the assumption of Theorem \ref{thm:mainThmC4T4} are that the presentation $\P$ satisfies the following two hypotheses:

\begin{enumerate}
	\item[(${\cal{H}}_1$)] All relators of $\P$ are of length four and cyclically reduced.
	\item[(${\cal{H}}_2$)] If $M$ is a minimal van Kampen diagram over the presentation $\P$ then:
	\begin{enumerate}
		\item $M$ is a $C(4) \& T(4)$ diagram.
		\item If $\rho$ is a path of $M$ which is on the boundary of two regions then $\rho$ is labelled by a generator.
	\end{enumerate}
\end{enumerate}

We will additionally assume that all diagrams have the property that outer edges are labelled by a generator (consequently, lengths of a paths and the length of their labels coincide). This assumption can be guaranteed by introducing vertices of valence two along the boundary. 

A key element of the proof is showing that the requirements of Proposition \ref{prop:falseKFT} hold. Thus, we need to construct an order ``$\prec$'' on $\ws{X}$ for which we need to show that for every non-minimal element $W \in \ws{X}$ we can find another element $V \in \ws{X}$ such that
\begin{enumerate}
 \item $V\prec W$;
 \item $V=W$ in $G$;
 \item $V$ and $W$ are $k$-fellow-travelers.
\end{enumerate}
For brevity, we would say in this case that $V$ $k$-refutes $W$ and also that $W$ can be refuted. A large part of the proof will be concerned with showing that the non-minimal elements of $\ws{X}$ can be refuted (according to an order we construct later). One of important characteristics of the order we construct is the property that if a word $V \in \ws{X}$ is shorter then a word $W \in \ws{X}$ then $V$ precedes $W$ in the order. Thus, we start with the following definition.

\begin{definition}[Shortable paths and words] \label{def:shortable}
Let $M$ be a minimal diagram over $\P$ with boundary label $\alpha\beta^{-1}$. We say that the path $\alpha$ is \emph{shortable} if there is a sub-diagram $N$ of $M$ with $\partial N = \mu\sigma^{-1}$ such that $\mu$ is a sub-path of $\alpha$, $|\mu|>|\sigma|$, and $N$ has connected interior and consists of at most two regions. A word $W \in \ws{X}$ is called \emph{shortable} if there is a a minimal diagram $M$ over $\P$ with boundary cycle $\alpha\beta^{-1}$ such that $\alpha$ is labelled by $W$ and $\alpha$ is shortable.
\end{definition}

Clearly, labels of shortable paths are not geodesic. Similarly, shortable words are not geodesics (since they are the label of a shortable path in some diagram).

\begin{lemma} \label{lem:shortableRefute}
If $W$ is shortable or is not freely-reduced then there is an element $V$ such that $|V|<|W|$, $V=W$ in $G$, and $V$ and $W$ are $6$-fellow-travellers.
\end{lemma}
\begin{proof}
Assume that $W$ is shortable. Let $M$ be a minimal diagram over $\P$ with boundary cycle $\alpha\beta^{-1}$ such that $\alpha$ is labelled by $W$ and $\alpha$ is shortable. Let $N$ a sub diagram of $M$ containing at most two regions with $\partial N = \mu\sigma^{-1}$, $\mu$ is a sub-path of $\alpha$, and $|\mu|>|\sigma|$. Note that since $N$ contains at most two regions we have $|\mu|+|\sigma|\leq6$ (since $N$ has connected interior and using hypothesis ${\cal{H}}_1$). Suppose that $\alpha = \alpha_1 \mu \alpha_2$ and $W = W_1 W_2 W_3$ where $\alpha_1$ is labelled by $W_1$, $\mu$ is labelled by $W_2$, and $\alpha_2$ is labelled by $W_3$. Let $V$ be the label of $\sigma$ and let $U = W_1 V W_3$. Then, $|V| < |W_2|$ and so $|U| < |W|$. Since $V = W_2$ in $G$ we also get that $U = W$ in $G$. Finally by Observation \ref{obs:CheckingFeloTrvlProp} we get that $U$ and $W$ and $6$-fellow-travellers. Similarly, if $W$ is not freely-reduced then we can write $W = W_1 xx^{-1} W_2$ for $x \in \pmX$. Let, $U = W_1 W_2$. The conclusion now follows along the same lines.
\end{proof}

We next remark on two implications of hypothesis ${\cal{H}}_2$.

\begin{figure}[ht]
\centering
\includegraphics[totalheight=0.1\textheight]{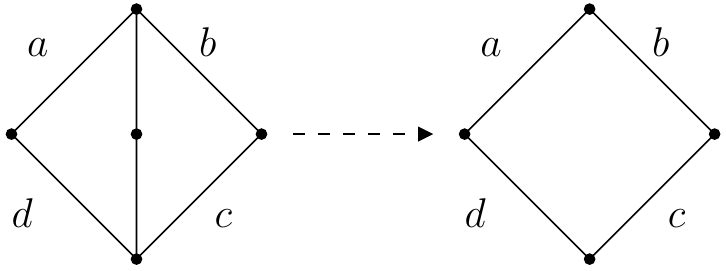}
\caption{Inner vertex of valence two and a possible fix}\label{fig:innVertValTwoAndFix}
\end{figure}

\begin{remark} \label{rem:twoTypesOfNonMin}
There are two examples of non-minimal diagrams (over the presentation $\P$) which occur frequently. In these cases the assumption is that the boundary label is freely-reduced and that the diagrams are reduced. 

The first examples is of a diagram with two regions and one inner vertex of valence two; see the left side of Figure \ref{fig:innVertValTwoAndFix}. By the second condition of hypothesis ${\cal{H}}_2$, paths in the boundary of two regions in a minimal diagram are labelled by a generator. But, in this case the inner path is labelled by two generators. Thus, the diagram is not minimal. Since we are assuming that the diagram is reduced it follows that the only way to reduce the number of regions is by replacing them with a single region; this is illustrated in the right side of Figure \ref{fig:innVertValTwoAndFix}.

The second examples is of a diagram with three regions and one inner vertex of valence three; see the top part of Figure \ref{fig:innVertValThreeAndFixes}. Due to the vertex of valence three, the diagram is not a $C(4)\&T(4)$ diagrams. It follows from the first condition of hypothesis ${\cal{H}}_2$ that this diagram is not minimal. In the bottom of the figure we illustrate three possible diagrams which have the same boundary label but with two regions (clearly, reducing the number of regions to one is impossible due to the length of the boundary). It is important to note that since we assumed that the boundary label of the top diagram is freely-reduced we get that one of these examples is a diagram that can be constructed over the presentation $\P$. Notice however that \emph{not} necessarily all of the three diagrams can be constructed over the presentation $\P$ (because there my not be enough relations in $\P$ to construct these diagrams).
\end{remark}

\begin{figure}[ht]
\centering
\includegraphics[totalheight=0.30\textheight]{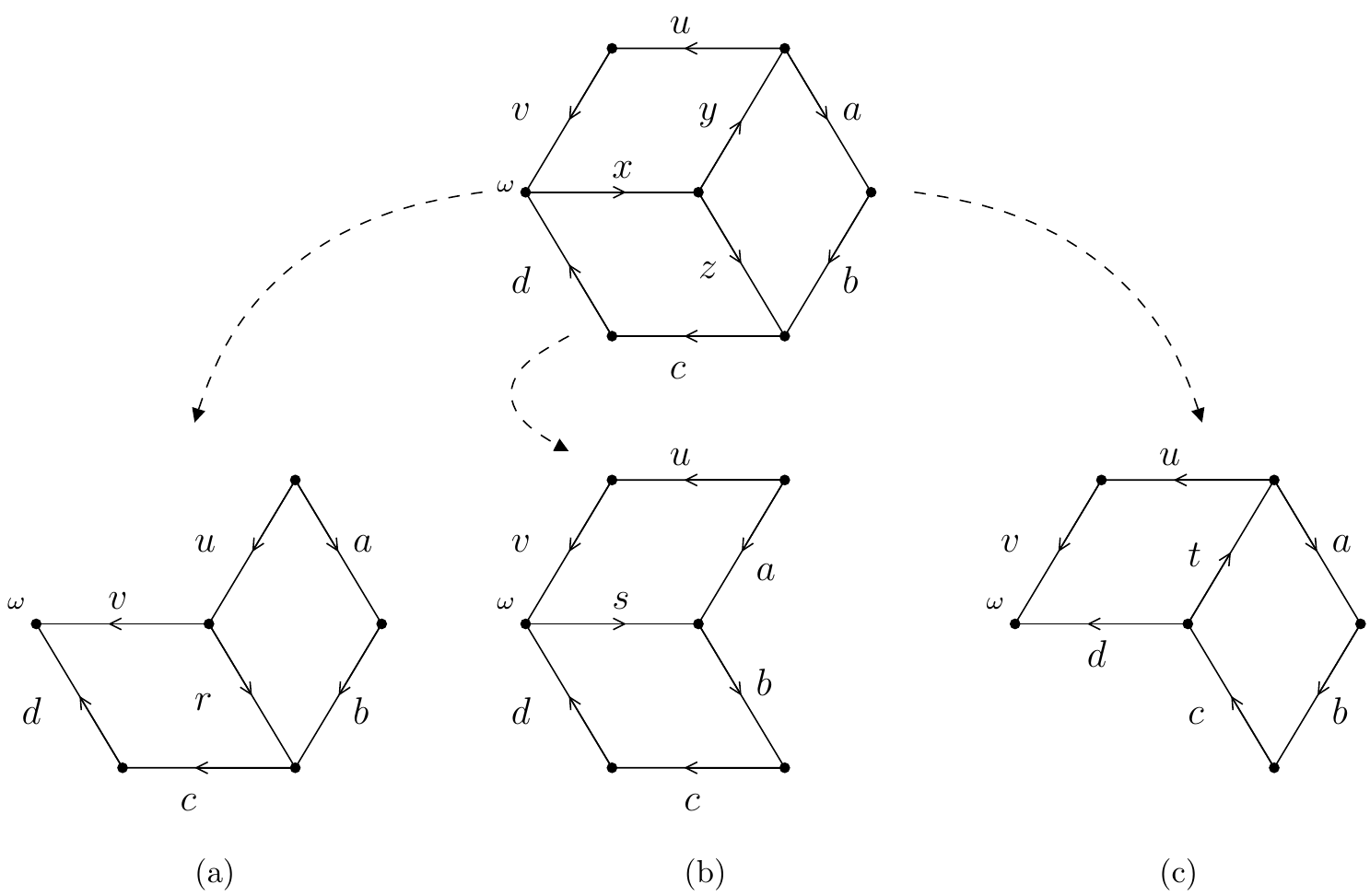}
\caption{Inner vertex of valence three and possible fixes}\label{fig:innVertValThreeAndFixes}
\end{figure}

The main difference between the presentation $\P$ and a presentation satisfying the algebraic $C(4) \& T(4)$ conditions lays in the fact that one can construct reduced diagrams with inner vertices of valence three. As shown in Figure \ref{fig:innVertValThreeAndFixes} there may be up to \emph{three ways} to reduce the number of regions in such diagrams. If there are \emph{at least two} different ways to minimize such a diagram then we get two diagrams that have the same boundary label but which are not equivalent (in the bottom of Figure \ref{fig:innVertValThreeAndFixes} we illustrate three diagrams with the same boundary label which are not equivalent). As we shall see later, it is important to understand this situation. We thus introduce the following terminology.

\begin{definition}[Domino diagram] \label{def:domino}
A \emph{domino diagram} is a diagram $M$ which contains two regions, its boundary is of length six, and there is a single inner edge. A \emph{biased domino} is a couple $(M,\mu)$ where $M$ is a domino diagram $M$ and $\mu$ is a boundary cycle of $M$. We will also say that the domino $M$ is biased through the boundary cycle $\mu$. Also, we will sometimes abuse the notation and say that $M$ is biased when $\mu$ is known from the context. A biased domino $(M,\mu)$ has three possible types: 
\begin{enumerate}
 \item \textbf{3-Typed}. When the first three edges of $\mu$ lay on one region. See Figure \ref{fig:oriDom}(a).
 \item \textbf{2-Typed}. When the first two edges of $\mu$ lay on one region but not the first three. See Figure \ref{fig:oriDom}(b).
 \item \textbf{1-Typed}. When the first edge of $\mu$ lays on one region but not the first two edges. See Figure \ref{fig:oriDom}(c).
\end{enumerate}
\end{definition}

\begin{figure}[ht]
\centering
\includegraphics[totalheight=0.15\textheight]{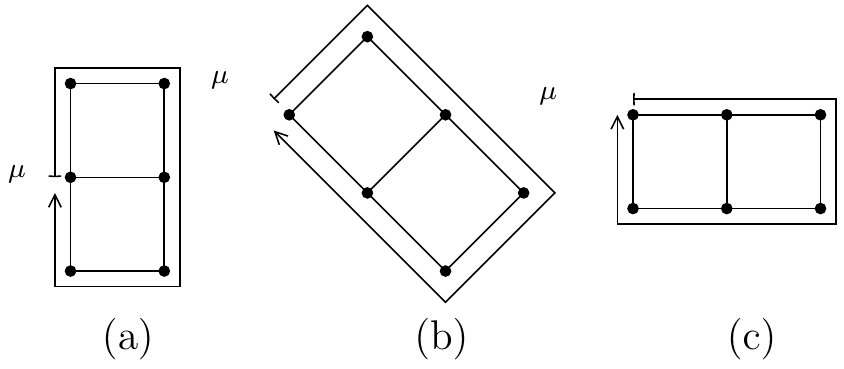}
\caption{The three types of biased dominos}\label{fig:oriDom}
\end{figure}

Next we define the important notion of a \emph{stable domino}. The idea is that a domino is stable if we cannot replace its interior without changing its boundary (see below for a rigorous definition). One reason why stable dominos are important comes from the fact that if the presentation $\P$ satisfies the algebraic $T(4)$ condition then all dominos are stable (this is an easy fact which follows from the definitions).

\begin{definition}[Stable Domino] \label{def:stableDomino}
Let $M$ be a domino diagram. $M$ is a \emph{stable domino} if any other domino with the same boundary label is equivalent to $M$ (see the discussion after Definition \ref{def:vanKampenDiagram}). In other words, if $\mu$ is a boundary cycle of $M$ and we are given a biased domino $(M',\mu')$ such that $\mu$ and $\mu'$ have the same label then we have that $(M,\mu)$ and $(M',\mu')$ have the same type. A domino which is not stable will be called \emph{unstable}.
\end{definition}

As an example, if two of the dominos at the bottom of Figure \ref{fig:innVertValThreeAndFixes} can be constructed over the presentation $\P$ then these two dominos are not stable (since they have the same boundary label but are not equivalent). On the other hand, if \emph{only one} of the dominos at the bottom of Figure \ref{fig:innVertValThreeAndFixes} can be constructed then this domino is stable. We emphasis this point since it may very well be that not all three diagrams on the bottom can be constructed over $\P$ (they are just drawn for illustration purposes)

We turn to define an order on the elements of $\ws{X}$. We begin by identifying triplets of generators which appear as labels of stable dominos.

\begin{definition}[Stable triplets] \label{def:stableTriplet}
Let $x,y,z \in \pmX$. We say that the triplet $(x,y,z)$ is stable if there is a \emph{stable} domino $N$ were its boundary cycle contains two consecutive edges $e_1$ and $e_2$ which are labelled by $x$ and $y$, respectively. Also, the terminal vertex of $e_1$ (which is also the initial vertex of $e_2$) is of valence three and the inner edge $e_3$ that emanate from it is labelled by $z$; see Figure \ref{fig:stableTriplet}.
\end{definition}

\begin{figure}[ht]
\centering
\includegraphics[totalheight=0.10\textheight]{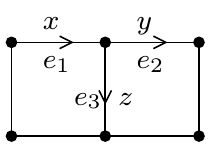}
\caption{Illustration of a stable triplet}\label{fig:stableTriplet}
\end{figure}

Using the definition of stable triplets we can attach a special graph to each generator in $\pmX$.

\begin{definition}[The Stability Graph $\Gamma_x$] \label{def:stabilityGraph}
Let $x \in \pmX$. We define the directed graph $\Gamma_x = (V_x, E_x)$ with vertex set $V_x$ and edge set $E_x$. The vertex set $V_x$ is the set $\pmX$. Let $y$ and $z$ be two elements of $V_x$ (i.e., two generators). There is a directed edge $(y,z) \in E_x$ (i.e., an edge from $y$ to $x$) if and only if $(x,y,z)$ is a stable triplet.
\end{definition}

We claim that the stability graph $\Gamma_x$ induces a linear order `$<_x$' on $\pmX$ such that if $(x,y,z)$ is a stable triplet then `$z<_x y$'. It will only be true if the graph $\Gamma_x$ contains no cycles. This is the content of the next proposition, whose proof is postponed to Sub-Section \ref{ssec:StabGrphNoCycle}.

\begin{proposition} \label{prop:StabGrphNoCycle}
Let $x \in \pmX$. The graph $\Gamma_x$ is cycle-free.
\end{proposition}
\begin{proof}
See Sub-Section \ref{ssec:StabGrphNoCycle}.
\end{proof}

\begin{corollary} \label{cor:linOrder}
There exist a linear order `$<_x$' defined on $\pmX$ such that if $(x,y,z)$ is a stable triplet for $y$ and $z$ in $\pmX$ then $z<_x y$.
\end{corollary}
\begin{proof}
This follows since a cycle-free directed graph induce a partial order on the set of vertices. This partial order can be completed into a linear order on the vertices.
\end{proof}

\begin{definition}[Grading function] \label{def:gradingFunc}
A grading functions for $x \in \pmX$ is a function $G_x:\pmX\to\Set{1,\ldots,2|X|}$ such that if $y$ and $z$ are in $\pmX$ and $(x,y,z)$ is a stable triplet then $G_x(z) < G_x(y)$.
\end{definition}

The following lemma is clear from Corollary \ref{cor:linOrder}.

\begin{lemma}
For each $x \in \pmX$ there is a grading function.
\end{lemma}

For the rest of the section we fix a set $\Set{G_x|x \in \pmX}$ of grading functions. Next, the order on $\ws{X}$.

\begin{definition}[Peifer vector]
Let $W = x_1 x_2 \cdots x_n \in \ws{X}$ be a word of length $n$. We assign a vector $\kappa_W = (a_1,a_2,\ldots,a_n) \in \mathbb{N}^n$ to $W$. The first entry $a_1$ is zero (i.e., $a_1 = 0$) and $a_i = G_{x_{i-1}}(x_i)$ for $1 < i \leq n$.
\end{definition}

There is a natural lexicographical order on the elements of $\mathbb{N}^*$, the finite vectors over the natural numbers. Also, by fixing some---arbitrary---order on $\pmX$ we can assign a lexicographical order on $\ws{X}$. We use these to define an order on $\ws{X}$.

\begin{definition}[Order ``$\prec$'' on $\ws{X}$]
Let $W$ and $U$ be two elements of $\ws{X}$. We say that $W \prec U$ if either:
\begin{enumerate}
	\item $\kappa_W$ precedes $\kappa_U$ in lexicographical order.
	\item $\kappa_W = \kappa_U$ and $W$ precedes $U$ in lexicographical order.
\end{enumerate}
\end{definition}

It is straightforward to show that the Peifer order is regular (Definition \ref{def:regOrder}). Thus, the proof of the following lemma is omitted.

\begin{lemma} \label{lem:precRegular}
The order ``$\prec$'' is regular.
\end{lemma}

Our next goal is to show that the conditions of Proposition \ref{prop:falseKFT} hold for the order ``$\prec$''. We start with two definitions: the definition of a domino that is contained in a path and the definition of a domino that is well-positioned.

\begin{definition}[Domino contained in a path] \label{def:contInPath}
Let $M$ be a diagram over the presentation $\P$ with boundary cycle $\alpha\beta^{-1}$ and let $N$ be a domino sub-diagram of $M$. Suppose that the boundary cycle of $N$ decomposes as $\partial N = \mu\sigma^{-1}$. We say that \emph{$N$ is contained in $\alpha$} (resp., in $\beta$) if $|\mu| = |\sigma| = 3$ and the following hold:
\begin{enumerate}
	\item $\partial N \cap \alpha = \mu$ (resp., $\partial N \cap \beta = \mu$).
	\item $\alpha = \alpha_1 \mu \alpha_2$ (resp., $\beta = \beta_1 \mu \beta_2$).
	\item The last edge of $\sigma$ is an inner edge (this is the fourth edge in the boundary cycle of $M$ which start with $\mu$).
\end{enumerate}
If the domino $N$ is contained in $\alpha$ (resp., $\beta$) then, using the above notation, it is by default made biased through the boundary cycle $\mu\sigma^{-1}$. 
\end{definition}

In words, a domino $N$ is contained in a boundary path $\alpha$ if exactly half of the boundary cycle $\partial N$ is contained in the path $\alpha$. Also, the last edge of boundary cycle which is not on the path $\alpha$ is an inner edge.

\begin{definition}[Well-positioned domino] \label{def:wellPosDominos}
Let $M$ be a diagram over the presentation $\P$ with boundary cycle $\alpha\beta^{-1}$ and let $N$ be a domino sub-diagram of $N$. Suppose that $N$ is made biased through a boundary cycle $\partial N = \mu\sigma^{-1}$ where $\mu = \alpha \cap \partial N$. Suppose further that the domino $N$ is contained in $\alpha$ (resp., $\beta$). We say that $N$ is \emph{well-positioned} in $\alpha$ (resp., in $\beta$) if either:
\begin{enumerate}
	\item\label{def:wellPosDominos:1} $N$ is $3$-typed or $2$-typed.
	\item\label{def:wellPosDominos:2} $N$ is $1$-typed and $N$ is a stable domino.
\end{enumerate}
\end{definition}

So, if a domino $N$ is $1$-typed and is well-positioned then it follows that it is also stable. Also, if a domino is contained in a path and is not well-positioned then it is $1$-typed and is not stable. With these definitions we can give the first result toward showing that non-``$\prec$''-minimal elements can be refuted.

\begin{proposition} \label{prop:thickConfToRefute}
Let $M$ be a minimal diagram with boundary cycle $\alpha\beta^{-1}$. Assume that all dominos contained in $\alpha$ are well-positioned. Let $W \in \ws{X}$ be the label of $\alpha$ and assume that $W$ is not shortable (Definition \ref{def:shortable}). If $\alpha$ contains a thick configuration (Definition \ref{def:thickConf}) then $W$ can be $4$-refuted.
\end{proposition}
\begin{proof}
Let $N$ be the thick configuration in $M$ that is contained in $\alpha$. Since $W$ is not shortable the thick configuration is of the second type. $N$ has connected interior and consists of two neighboring regions, $D_1$ and $D_2$, such that $D_2$ has two inner edges and two outer edges; see Figure \ref{fig:refSecType}. Since $W$ is not shortable, the region $D_1$ has one outer edge in $\alpha$. Assume that $\partial D_2 = \delta\rho^{-1}$ where $\delta = \partial D_2 \cap \alpha$ is the outer boundary of $D_2$ and $\rho$ is the inner boundary. Let $v_1$ the boundary vertex that belong to both $\partial D_1$ and $\partial D_2$. Then, $v_1$ is of valence three. Let $e_1^{-1}$, $e_2$, and $e_3$ be the three edges which emanate from $v_1$: $e_1$ is an edge of $D_1$, $e_2$ is an inner edge, and $e_3$ is an edge of $D_2$. Let $v_2$ be the other vertex of $e_1$. By hypothesis ${\cal{H}}_1$ every relator is of length four so we get that $N$ is a domino. We make $N$ biased through the boundary cycle that begins in the edge $e_1$. In this case $N$ is $1$-typed and the domino $N$ is contained in $\alpha$ (the three outer edges of $N$ are contained in $\alpha$ and the last edge of $\partial N \setminus \alpha$ is an inner edge). Since by assumption $N$ is well-positioned we get that $N$ is stable (see Part \ref{def:wellPosDominos:2} of Definition \ref{def:wellPosDominos}).

Now, assume that $\mu$ is labelled by $xya$ and that $\rho$ is labelled by $zb$ ($x$, $y$, $z$, $a$, and $b$ are elements of $\pmX$). Decompose $W$ as $W = W_1 xya W_2$ and let $U = W_1 xzb W_2$. Notice that $ya = zb$ in $G$ (because $yab^{-1}z^{-1}$ is the boundary label of $D_2$) so $U = W$ in $G$. Notice also that $U$ and $W$ are $4$-fellow-travellers (Observation \ref{obs:CheckingFeloTrvlProp}). Since $N$ is a stable domino we get that $(x,y,z)$ is a stable triplet (Definition \ref{def:stableTriplet}). Hence by the definition of a grading function (Definition \ref{def:gradingFunc}) we get that $G_x(z) < G_x(y)$. Consequently, we get that the $|W_1|+1$-th coordinate for $\kappa_U$ is strictly less than the $|W_1|+1$-th coordinate for $\kappa_W$. The coordinates before the $(|W_1|+1)$-th one in $\kappa_W$ and in $\kappa_U$ are the same. Hence, we established that $U \prec W$ (because $\kappa_U$ precedes $\kappa_W$ in lexicographical order). Consequently, $U$ $4$-refute $W$.
\end{proof}

\begin{figure}[ht]
\centering
\includegraphics[totalheight=0.15\textheight]{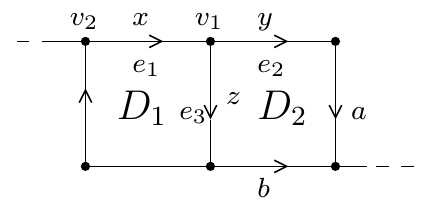}
\caption{Refuting through a thick configuration of second type}\label{fig:refSecType}
\end{figure}

The following proposition, whose proof is postponed to Sub-Section \ref{ssec:diagDomInTopPos}, shows that we can usually assume that all dominos contained in a path are well-positioned.

\begin{proposition} \label{prop:diagDomInTopPos}
Let $W,U \in \ws{X}$ and $x,y \in \pmX \cup \Set{\eps}$ such that $xW = Uy$ in $G$ and assume that $W$ and $U$ are freely-reduced and not shortable. Then, there is a minimal diagram $M$ with boundary cycle $\sigma \alpha \tau^{-1} \beta^{-1}$ such that $\alpha$ is labelled by $W$, $\beta$ is labelled by $U$, $\sigma$ is labelled by $x$, and $\tau$ is labelled by $y$ ($\sigma$ is empty if $x=\eps$ and similarly $\tau$ is empty if $y=\eps$) such that all dominos contained in $\alpha$ and in $\beta$ are well-positioned.
\end{proposition}
\begin{proof}
See Sub-Section \ref{ssec:diagDomInTopPos}.
\end{proof}

Using Proposition \ref{prop:diagDomInTopPos} we next show how to refute non-``$\prec$''-minimal elements.

\begin{lemma} \label{lem:nonMinCanRef}
If $W$ is not ``$\prec$''-minimal then we can $6$-refute $W$.
\end{lemma}
\begin{proof}
By Lemma \ref{lem:shortableRefute} we can assume that $W$ is not shortable and freely-reduced. Let $U$ be ``$\prec$''-minimal such that $W=U$ in $G$. Again by Lemma \ref{lem:shortableRefute}, since $U$ is ``$\prec$''-minimal it is not shortable and freely-reduced. Thus, by Proposition \ref{prop:diagDomInTopPos} there is a minimal diagram $M$ over $\P$ with boundary cycle $\alpha\beta^{-1}$ such that $\alpha$ is labelled by $W$, $\beta$ is labelled by $U$, and all dominos in both paths are well-positioned. By Proposition \ref{prop:thickConfToRefute} the path $\beta$ does not contain thick configurations (since $U$ is ``$\prec$''-minimal). If the path $\alpha$ contain a thick configuration then by Proposition \ref{prop:thickConfToRefute} we can $4$-refute $W$. Thus, we can assume that also the path $\alpha$ does not contain thick configurations. It follows therefore by Theorem \ref{thm:noThickToThinDiag} that $M$ is an $(\alpha,\beta)$-thin-diagram. By Lemma \ref{lem:thinDiagForKFT} either $W$ and $U$ are $4$-fellow-travellers or there is $W' \in \ws{X}$ such that $|W'|<|W|$, $W' = W$ in $G$, and $W'$ and $W$ are $4$-fellow-traveller. In the first case $U$ $4$-refute $W$ and in the second case $W'$ $4$-refute $W$ (since $|W'|<|W|$ implies that $W' \prec W$). Thus, $W$ can be $6$-refuted.
\end{proof}

To satisfy the conditions of Proposition \ref{prop:falseKFT} we need a regular order where each non-minimal element can be refuted and such that there is a minimal element for every non-empty subset of $\ws{X}$. These requirements are satisfied for the order ``$\prec$'': by Lemma \ref{lem:precRegular} the order ``$\prec$'' is regular; by Lemma \ref{lem:nonMinCanRef} each non-minimal element can be refuted; and, since lexicographical order is a well-order each non-empty subset of $\ws{X}$ has a minimal element. Consequently we get the following result:

\begin{proposition} \label{prop:precMinIsReg}
Let $L$ be the set of ``$\prec$''-minimal elements. Then, $L$ is regular and maps onto $G$ through the natural map.
\end{proposition}

We are now ready to complete the proof of Theorem \ref{thm:mainThmC4T4}.

\begin{proof}[Proof of Theorem \ref{thm:mainThmC4T4}]
Let $L$ be the set of ``$\prec$''-minimal elements. It follows from Proposition \ref{prop:precMinIsReg} that $L$ is a regular set and maps onto $G$ though the natural map. To complete the proof of the theorem by showing that $L$ has the fellow-traveller property (see Definition \ref{def:FellowTraveller}) and thus it is a bi-automatic structure of $G$. Let $W$ and $U$ be two elements of $L$ and assume that there are $x,y \in \pmX \cup \Set{\eps}$ such that $xW = Uy$ in $G$. By Proposition \ref{prop:diagDomInTopPos} there is a minimal diagram $M$ with boundary cycle $\sigma\alpha\tau^{-1}\beta$ such that $\alpha$ is labelled by $W$, $\beta$ is labelled by $U$, $\sigma$ is labelled by $x$, $\tau$ is labelled by $y$, and all dominos in $\alpha$ and in $\beta$ are well-positioned. By Proposition \ref{prop:thickConfToRefute} there are no thick configuration in $\alpha$ or in $\beta$ and thus by Theorem \ref{thm:noThickToThinDiag} the diagram $M$ is $(\sigma\alpha,\beta\tau)$-thin diagram. Now, since $W$ and $U$ are ``$\prec$''-minimal they are geodesics and consequently the conditions of Lemma \ref{lem:husdDist} are satisfied where the constant $s$ can be taken to be $4$ (see Lemma \ref{lem:thinDiagToHaudDist}). Consequently, $xW$ and $Uy$ are $9$-fellow-travellers. This show that $L$ has the fellow-traveller property and the proof of the theorem is completed.
\end{proof}

\subsection{Well-Positioning dominos \label{ssec:diagDomInTopPos}}

In this sub-section we prove Proposition \ref{prop:diagDomInTopPos}. The reader may want to recall the definition of a domino and a biased domino (Definition \ref{def:domino}) and also the definition of a stable domino (Definition \ref{def:stableDomino}). Recall that we are given $W,U \in \ws{X}$ and $x,y \in \pmX \cup \Set{\eps}$ such that $xW = Uy$ in $G$ and the assumption is that $W$ and $U$ are freely-reduced and not shortable (Definition \ref{def:shortable}). We need to show that there is a minimal diagram $M$ with boundary cycle $\sigma \alpha \tau^{-1} \beta^{-1}$ such that $\alpha$ is labelled by $W$, $\beta$ is labelled by $U$, $\sigma$ is labelled by $x$, and $\tau$ is labelled by $y$ ($\sigma$ is empty of $x=\eps$ and similarly $\tau$ is empty if $y=\eps$) such that all dominos contained in $\alpha$ and in $\beta$ are well-positioned. We repeat two of the needed definitions here:
\begin{enumerate}
	\item Domino contained in a path (Definition \ref{def:contInPath}): a domino $M$ is contained in a path $\alpha$ if $\mu=\partial M \cap \alpha$ contains exactly half of $\partial M$ and the next edge in $\partial M$ after $\mu$ is an inner edge.
	\item Well-positioned domino (Definition \ref{def:wellPosDominos}): a domino $M$ is well-positioned in a path $\alpha$ if it is contained in $\alpha$ (this automatically makes $M$ a biased domino by the boundary cycle that start with $\alpha \cap \partial M$) and it is 1-typed (see definition \ref{def:domino}) only when it is stable (see definition \ref{def:stableDomino}).
\end{enumerate}

Naively, one would take any minimal diagram with a needed boundary label and simply change the interior of dominos which are not well-positioned. However, this approach is problematic since there may be overlapping dominos (so, fixing one domino may introduce other dominos which are not well-positioned). Thus, we start with an analysis of how two dominos can overlap.

We remark that if $N$ is a domino that is contained in a path $\alpha$ and the path is not shortable then the domino cannot be $3$-typed (because you cannot have three consecutive edges of a single region in $\alpha$). Also, if a domino $N$ that is contained in a path $\alpha$ is not well-positioned then it is $1$-typed and unstable. It follows that if we assume that the label of the path $\alpha$ is not shortable then the domino $N$ can be only be replaced with a domino $N'$ which is $2$-typed.

Using the notation of Definition \ref{def:contInPath}, if $N$ is a domino that is contained in $\alpha$ (resp., in $\beta$) then it make sense to order the regions in $N$. The region which its boundary contains the first edge of $\mu$ will be called the \emph{first region} and the other region (the one that its boundary contains the last edge of $\sigma$) will be called the \emph{second region}.

\begin{lemma} \label{lem:noOverlapInDiffPaths}
Let $M$ be a minimal diagram over $\P$ with boundary label $\sigma\alpha\tau^{-1}\beta^{-1}$, connected interior, and no spikes. Assume that the paths $\alpha$ and $\beta$ are not shortable. Let $N_1$ and $N_2$ be two domino sub-diagrams of $M$ which are contained in $\alpha$ and $\beta$, respectively. Assume that $\partial N_1 = \mu_1 \rho_1^{-1}$ and $\partial N_2 = \mu_2 \rho_2^{-1}$ where $\mu_1 = \alpha \cap \partial N_1$ and $\mu_2 = \beta \cap \partial N_2$. If the interior of $N_1$ and $N_2$ are not disjoint then the following hold:
\begin{enumerate}
	\item $\mu_1$ is a prefix of $\alpha$ and $\mu_2$ is a prefix of $\beta$.
	\item $\sigma$ is an empty path.
	\item The first region of $N_1$ and the first region of $N_2$ are equal.
\end{enumerate}
\end{lemma}
\begin{proof}
Let $D_1$ and $D_2$ be the first and second regions of $N_1$. Similarly, let $E_1$ and $E_2$ be the first and second regions of $N_2$. First notice that we cannot have that $N_1 = N_2$ since then the last edges of $\rho_1$ and $\rho_2$ (which in this case are equal to $\mu_2$ and $\mu_1$, respectively) will not be inner edges of $M$. The case that $D_2 = E_1$ is impossible, as illustrated in Figure \ref{fig:overlapDominos}(a) and Figure \ref{fig:overlapDominos}(b), since either $\alpha$ is shortable or the first edge of $\mu_2$ is an inner edge. Similarly, we cannot have that $D_1 = E_2$. Note that in general, since both $\alpha$ and $\beta$ are not shortable we cannot have that $\mu_1$ or $\mu_2$ would be the outer boundary of a single region (like $\mu_1$ in Figure \ref{fig:overlapDominos}(a)). The case that $D_2 = E_2$ is also impossible since the last edge of $\rho_1$ and the last edge of $\rho_2$ are not inner edges of $M$; see Figure \ref{fig:overlapDominos}(c). Thus, we are left with the case that $D_1 = E_1$ and $D_2 \neq E_2$, as illustrated in Figure \ref{fig:overlapDominos}(d). This also shows that the initial vertices of $\mu_1$ and $\mu_2$ are equal, i.e., $i(\mu_1)=i(\mu_2)$. Consequently, using the assumption that $M$ has connected interior and no spikes, we get that $\sigma$ is empty, $\mu_1$ is a prefix of $\alpha$, and $\mu_2$ is a prefix of $\beta$.

\begin{figure}[ht]
\centering
\includegraphics[totalheight=0.15\textheight]{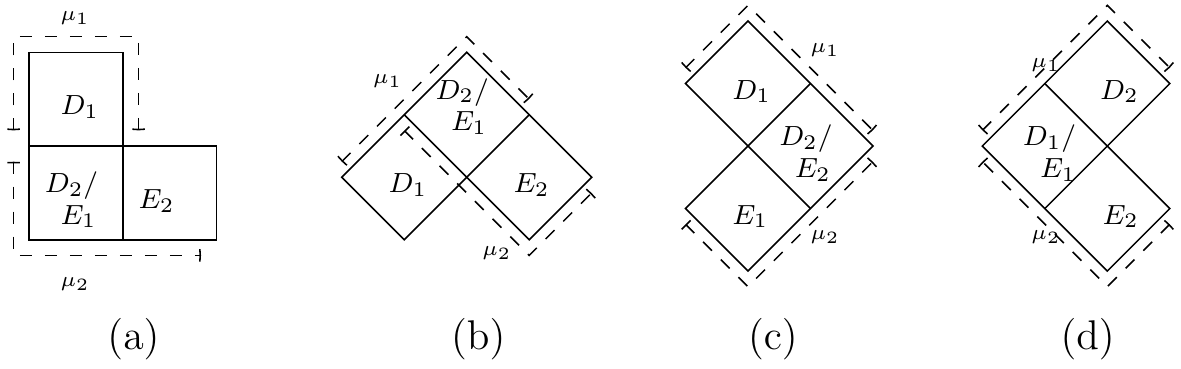}
\caption{Possible overlaps between dominos}\label{fig:overlapDominos}
\end{figure}

\end{proof}

Next lemma states that if we fix an unstable $1$-typed domino $N$ then we never form a new domino to the \emph{right} of $N$ (where we think on $\alpha$ as going from left to right).

\begin{lemma} \label{lem:noOverlapInSamePath}
Let $M$ be a minimal diagram over $\P$ with boundary label $\alpha\beta^{-1}$, connected interior, and no spikes. Assume that the path $\alpha$ and its label are not shortable. Let $N$ be a domino sub-diagram of $M$ which is contained in $\alpha$. Assume further that $\partial N = \mu \sigma^{-1}$ where $\mu$ is a sub-path of $\alpha$. Finally, assume that $N$ is biased through $\mu \sigma^{-1}$, that $N$ is $1$-typed, and that $N$ is not stable. Denote the first and second regions of $N$ by $D_1$ and $D_2$, respectively. Let $D_3$ be a boundary region which its boundary contain the terminal vertex $t(\mu)$. Since $N$ is not stable there is a domino $N'$ with boundary cycle $\partial N' = \mu'(\sigma')^{-1}$ which has the same label as the label of $\mu \sigma^{-1}$ but $\mu'$ is $2$-typed. Suppose we form the diagram $M'$ by replacing $N$ with $N'$ in the diagram $M$. Then, the region $D_3$ is not the last region of a domino in $M'$ which is contained in $\alpha$.
\end{lemma}

\begin{figure}[ht]
\centering
\includegraphics[totalheight=0.10\textheight]{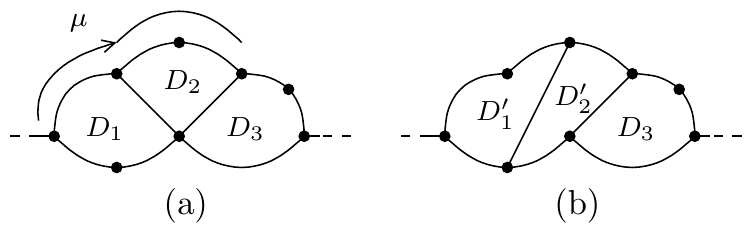}
\caption{An overlaps between dominos in the same path}\label{fig:overlapDominos2}
\end{figure}

\begin{proof}
First notice that since $N$ is $1$-typed we have that $D_2$ has two outer edges in $\alpha$, i.e., $|\partial D_2 \cap \alpha | = 2$. Suppose by contradiction that in the diagram $M'$ the region $D_3$ is the last region of a domino in $\alpha$ (in this case the terminal vertex $t(\mu)$ is of valence three). Let $D_1'$ and $D_2'$ be the regions of $N'$; see Figure \ref{fig:overlapDominos2}(b). Then, the domino that $D_3$ is its last region consists of the regions $D_2'$ and $D_3$. Since $N'$ is $2$-typed we get that $|\partial D_2' \cap \alpha | = 1$ so by the fact that $D_2'$ and $D_3$ form a domino in $M'$ that is contained in $\alpha$ we get that $|\partial D_3 \cap \alpha | = 2$. This shows that $|\partial (D_2 \cup D_3) \cap \alpha | = 4$ (see Figure \ref{fig:overlapDominos2}(a)) and consequently $\alpha$ is shortable in contradiction to our assumptions.
\end{proof}

Using the notation of Lemma \ref{lem:noOverlapInSamePath}, if $\alpha = \alpha_1 \mu \alpha_2$ then after replacing $N$ with $N'$ there is no domino in the path $\mu' \alpha_2$ (namely, to the right of $N'$) which its interior intersects the interior of $N'$. We use this for the inductive prove of the next lemma, which is as an intermediate step before the final proof of Proposition \ref{prop:diagDomInTopPos}.

\begin{lemma} \label{lem:interMidStep}
Assume the conditions of Proposition \ref{prop:diagDomInTopPos} and suppose that $W=W_1 W_2$ and $U=U_1U_2$ where $|W_2|+|U_2| < |W|+|U|$. Then, there exists a minimal diagram $M$ with boundary cycle $\sigma \alpha \tau^{-1} \beta^{-1}$ such that:
\begin{enumerate}
	\item $\alpha$ is labelled by $W$, $\beta$ is labelled by $U$, $\sigma$ is labelled by $x$, and $\tau$ is labelled by $y$.
	\item The path $\alpha$ decomposes as $\alpha = \alpha_1 \alpha_2$ and $\alpha_i$ is labelled by $W_i$ for $i=1,2$. In addition, all dominos contained in $\alpha_2$ are well-positioned.
	\item The path $\beta$ decomposes as $\beta = \beta_1 \beta_2$ and $\beta_i$ is labelled by $U_i$ for $i=1,2$. In addition, all dominos contained in $\beta_2$ are well-positioned.
\end{enumerate}
\end{lemma}

The statement of Lemma \ref{lem:interMidStep} is almost identical to the statement of Proposition \ref{prop:diagDomInTopPos} with the exception that we only require that dominos will be well-positioned if they are contained in the parts of $\alpha$ and $\beta$ which are labelled by $W_2$ or $U_2$. The proof follows. Before we give the proof we remark that it is enough to prove the lemma for connected components of the interior of the diagram $M$. The reason is that the lemma would hold if it holds for each of the component of the interior. Thus, the general case follows from the the case that the interior is connected and the diagram has no spikes.

\begin{proof}[Proof of Lemma \ref{lem:interMidStep}]
As remarked above, we will assume that the diagrams have connected interior and no spikes. The proof is by induction on $|W_2|+|U_2|$. The base case where $|W_2|+|U_2| \leq 1$ follows trivially. Assume that the claim is true for $|W_2|+|U_2| = n$ and consider the case where $|W_2|+|U_2| = n + 1 < |W|+|U|$. Assume w.l.o.g.\ that $|W_2|\geq2$ and write $W_2 = z W_2'$ where $z \in \pmX$. Since, $|W_2'| + |U_2| = n$ we get that by induction hypothesis there is a minimal diagram $M$ with boundary cycle $\sigma \alpha \tau^{-1} \beta^{-1}$ for which the following conditions hold:
\begin{enumerate}
	\item $\alpha$ is labelled by $W$, $\beta$ is labelled by $U$, $\sigma$ is labelled by $x$, and $\tau$ is labelled by $y$.
	\item The path $\alpha$ decomposes as $\alpha = \alpha_1 e \alpha_2$ where $\alpha_1$ is labelled by $W_1$, $e$ is labelled by $z$, and $\alpha_2$ is labelled by $W_2'$. Also, all dominos contained in $\alpha_2$ are well-positioned.
	\item The path $\beta$ decomposes as $\beta = \beta_1 \beta_2$ and $\beta_i$ is labelled by $U_i$ for $i=1,2$. In addition, all dominos contained in $\beta_2$ are well-positioned.
\end{enumerate}
We are done if all dominos in $e\alpha_2$ and $\beta_2$ are well-positioned. So, assume that there is a domino $N$ which is contained in $e\alpha_2$ which is not well-positioned (in this case the boundary cycle of $N$ must contains the edge $e$). Suppose we generate the diagram $M'$ from the diagram $M$ by replacing the domino $N$ with another domino which is well-positioned (i.e., by replacing an unstable $1$-typed domino with a $2$-typed domino). By Lemma \ref{lem:noOverlapInSamePath} all dominos of $M'$ which are contained in $e\alpha_2$ will be well-positioned (since the lemma say that there is no overlap between the domino we replaced and the rest of the dominos contained in $\alpha_2$). By Lemma \ref{lem:noOverlapInDiffPaths} the dominos of $M'$ which are contained in $\beta_2$ do not overlap with the dominos in $e\alpha_2$ (using the assumption that $|W_2| + |U_2|<|W|+|U|$ so either $\alpha_2 \neq \alpha$ or $\beta_2 \neq \beta$). Thus, all the dominos in $\beta_2$ are well-positioned. Consequently, the diagram $M'$ has all the needed properties and the claim is proved.
\end{proof}

Next the proof of the proposition.

\begin{proof}[Proof of Proposition \ref{prop:diagDomInTopPos}]
Since the case that $W=U=\eps$ is trivial, assume w.l.o.g.\ that $|W|\geq1$ and let $W = z W'$ for $z \in \pmX$. We have that $|W'|+|U| < |W|+|U|$ and thus by Lemma \ref{lem:interMidStep} there exists a minimal diagram $M$ with boundary cycle $\sigma \alpha \tau^{-1} \beta^{-1}$ such that:
\begin{enumerate}
	\item $\alpha$ is labelled by $W$, $\beta$ is labelled by $U$, $\sigma$ is labelled by $x$, and $\tau$ is labelled by $y$.
	\item The path $\alpha$ decomposes as $\alpha = e\alpha'$ where $e$ is an edge labelled by $z$ and $\alpha'$ is labelled by $W'$.
	\item All dominos contained in $\alpha'$ and $\beta$ are well-positioned.
\end{enumerate}
We are done if all dominos contained in $\alpha$ are well-positioned. Thus, suppose that there is a domino $N$ contained in $\alpha$ which is not well-positioned. Namely, $N$ is an unstable $1$-type domino. In this case the boundary cycle of $N$ must contain the edge $e$. We generate a new diagram $M'$ by replacing the domino $N$ with a well-positioned domino $N'$ (i.e., $N'$ is a $2$-typed domino). We claim that the the new diagram $M'$ has all the needed properties. This is obvious if the interior of $N$ does not overlap with a domino that is contained in $\beta$. Thus, assume that there is a domino $Q$ which is contained in $\beta$ and which overlap with the domino $N$ (in this case $\sigma$ is empty). Denote the first and second regions of $N$ by $D_1$ and $D_2$, respectively, and the last region of $Q$ by $D_3$. Recall that by Lemma \ref{lem:noOverlapInDiffPaths} we have that $D_1$ is also the first region of $Q$; see Figure \ref{fig:overlapAtBegining}(a). Denote the first and second regions of $N'$ by $D_1'$ and $D_2'$, respectively; see Figure \ref{fig:overlapAtBegining}(b). The last edge of $\beta$ in $M'$ is not part of the boundary cycle of a domino which is contained in $\beta$. Thus, all the dominos of $M'$ which are contained in $\beta$ are well-positioned (since they are well-positioned in $M$). By Lemma \ref{lem:noOverlapInSamePath} all dominos of $M'$ that are contained in $\alpha$ are well-positioned since they do not overlap with $N'$. Thus, all dominos in $M'$ which are contained in $\alpha$ and $\beta$ are well-positioned and the proposition is proved.

\begin{figure}[ht]
\centering
\includegraphics[totalheight=0.15\textheight]{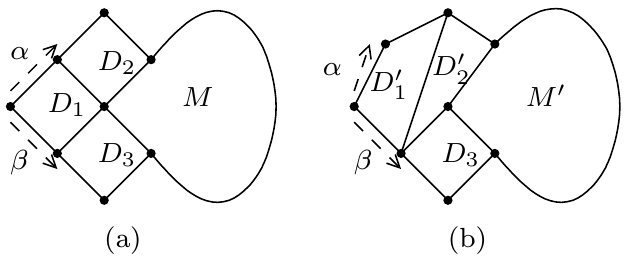}
\caption{An overlaps between dominos at the beginning of the paths}\label{fig:overlapAtBegining}
\end{figure}

\end{proof}

\subsection{Stability graph is cycle-free \label{ssec:StabGrphNoCycle}}

In this subsection we prove Proposition \ref{prop:StabGrphNoCycle}, namely that the stability graph is cycle-free. We start with a simple lemma.

\begin{lemma} \label{lem:stableDom}
Let $N$ be a stable domino with boundary cycle $\mu = \mu_1 \mu_2$. Suppose that $\mu_1$ is labelled by $xy$, that $\mu_2$ is labelled by $abcd$, and that $xy=zw$ in $G$ (all letters are generators in $\pmX$ and the boundary label is freely-reduced). See the left side of Figure \ref{fig:stableDom}. Then, there is a domino $N'$ with boundary cycle $\mu'$ such that $\mu'$ is labelled by $zwabcd$ (see the upper right side of Figure \ref{fig:stableDom}). Moreover, the biased dominos $(N,\mu)$ and $(N',\mu')$ have the same type.
\end{lemma}
\begin{proof}
First, since there is a relation $xyw^{-1}z^{-1}$ we can form a diagram with inner vertex of valence three with boundary label $zwabcd$ (see the left side of Figure \ref{fig:stableDom}). This diagram is not a $C(4)\&T(4)$ diagram (due to the inner vertex of valence three) and so it is not minimal. Consequently, there is a domino $N'$ with boundary cycle $\mu'$ such that $\mu'$ is labelled by $zwabcd$ (as explained in Remark \ref{rem:twoTypesOfNonMin}). If the biased dominos $(N,\mu)$ and $(N',\mu')$ have different type then the two letters $zw$ appear on the boundary of the same region in $N'$ (as in the lower right side of Figure \ref{fig:stableDom}). Consequently, we can form a biased domino $(N'',\mu'')$ where $\mu''$ has the label $xyabcd$ and such that $(N,\mu)$ and $(N'',\mu'')$ have different type. This would contradict the fact that $N$ is stable.
\end{proof}

\begin{figure}[ht]
\centering
\includegraphics[totalheight=0.25\textheight]{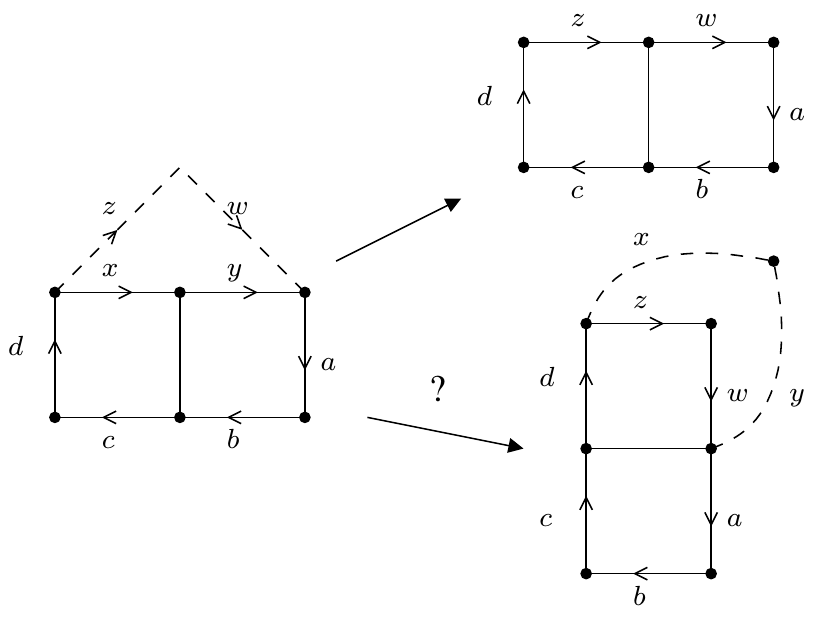}
\caption{A stable domino and an inner vertex of valence three}\label{fig:stableDom}
\end{figure}

\begin{remark}
Note that in the previous lemma, we can't have that $d=z^{-1}$ or $a=w^{-1}$. The reason is that if $d=z^{-1}$ then $abcw = abcdzw = abcdxy = 1$ and $xyw^{-1}d = xyw^{-1}z^{-1} = 1$ in $G$ so $abcw$ and $xyw^{-1}d$ are relators in $\P$. Thus, we can construct a domino with boundary label $xyabcd$ which is not equivalent to the domino $N$ and that would contradict the stability of $N$. We are using here a standard argument regarding proper $C(4)\& T(4)$ diagrams (which can be found, for example, in \cite{LS77}) stating that such diagram with more one region must have boundary cycle of length larger than four. Thus, if $abcd=1$ in $G$ then there is a minimal $C(4)\& T(4)$ diagram with boundary cycle labelled by $abcd$ and consequently this diagram must contain just one region.
\end{remark}

Next the proof of the proposition. For clarity we prove that there are no cycles of length four. It is straightforward to generalize this proof to show that there are no cycles of any length.

\begin{proof}[Proof of Proposition \ref{prop:StabGrphNoCycle}]
Let $x \in \pmX$ and suppose that there is a series of edges in $E_x$
\[
y_1 \to y_2 \to y_3 \to y_4
\]
where $y_1$, $y_2$, $y_3$, and $y_4$ are in $V_x$. We claim that there is no edge in $E_x$ from $y_4$ to $y_1$. The edges $(y_1,y_2)$, $(y_2,y_3)$, and $(y_3,y_4)$ in $E_x$ imply that the triplets $(x,y_1,y_2)$, $(x,y_2,y_3)$, and $(x,y_3,y_4)$ are stable (see Definition \ref{def:stabilityGraph} of the stability graph $\Gamma_x$). Hence, by the definition of stable triplet (Definition \ref{def:stableTriplet}) there are three stable biased domino $(N_1,\mu_1)$, $(N_2,\mu_2)$, and $(N_3,\mu_3)$ such that for $i=1,2,3$
\begin{enumerate}
	\item $\mu_i$ is labelled by $x y_i a_i b_i c_i d_i$.
	\item The inner edge of $N_i$ is labelled by $y_{i+1}$.
\end{enumerate}
See the left side of Figure \ref{fig:noCycles}. Now, assume by that there is another biased domino $(N_4,\mu_4)$ where $\mu_4$ labelled by $x y_4 a_4 b_4 c_4 d_4$ and the inner edge labelled by $y_1$ (see the lower left side of Figure \ref{fig:noCycles}). We construct a biased domino $(Q,\rho)$ (the lower right side of Figure \ref{fig:noCycles}) where $\rho$ is also labelled by $x y_4 a_4 b_4 c_4 d_4$ but the two biased dominos $(N_4,\mu_4)$ and $(Q,\rho)$ have different type (i.e., they are not equivalent). This would show that the domino $N_4$ cannot be stable and consequently, the triplet $(x,y_4,y_1)$ is not stable so there is no edge $(y_4,y_1)$ in $E_x$.

The first region of $N_4$ is labelled by $d_4^{-1} c_4^{-1} y_1^{-1} x$. Hence, we have that $d_4^{-1} c_4^{-1} = xy_1$ in $G$ and so by Lemma \ref{lem:stableDom} we have that there is a biased domino $(M_1,\sigma_1)$ where $\sigma_1$ is labelled by $d_4^{-1} c_4^{-1} a_1 b_1 c_1 d_1$ and the two biased dominos $(N_1,\mu_1)$ and $(M_1,\sigma_1)$ have the same type. Suppose the inner edge of $M_1$ is labelled by $m_1$. The domino $M_1$ contains a region with boundary label $d_4^{-1} m_1 c_1 d_1$ (see the upper right part of Figure \ref{fig:noCycles}) so $d_4^{-1} m_1 = d_1^{-1} c_1^{-1}$ in $G$. The domino $N_1$ contain a region labelled by $xyc_1d_1$ so $d_1^{-1} c_1^{-1} = xy_2$ and thus $d_4^{-1} m_1 = xy_2$ in $G$. We next repeat the argument for the domino $N_2$. Since $d_4^{-1} m_1 = xy_2$ in $G$ we get that by Lemma \ref{lem:stableDom} there is a biased domino $(M_2,\sigma_2)$ where $\sigma_2$ is labelled by $d_4^{-1} m_1^{-1} a_2 b_2 c_2 d_2$ and the two biased dominos $(N_2,\mu_2)$ and $(M_2,\sigma_2)$ have the same type. Suppose the inner edge of $M_2$ is labelled by $m_2 \in \pmX$ then we have that $d_4^{-1} m_2 = xy_2$ in $G$. Finally, repeating the same process we get that there is a biased domino $(M_3,\sigma_3)$ where $\sigma_3$ is labelled by $d_4^{-1} m_2^{-1} a_3 b_3 c_3 d_3$ and the two biased dominos $(N_3,\mu_3)$ and $(M_3,\sigma_3)$ have the same type. Also, if the inner edge of $M_3$ is labelled by some $m_3 \in \pmX$ then $d_4^{-1} m_3 = xy_4$ in $G$. This shows that $d_4^{-1} m_3 y_4^{-1} x^{-1}$ is a relation in $\RR$ (see the top region of the domino $Q$). 

\begin{figure}[ht]
\centering
\includegraphics[totalheight=0.35\textheight]{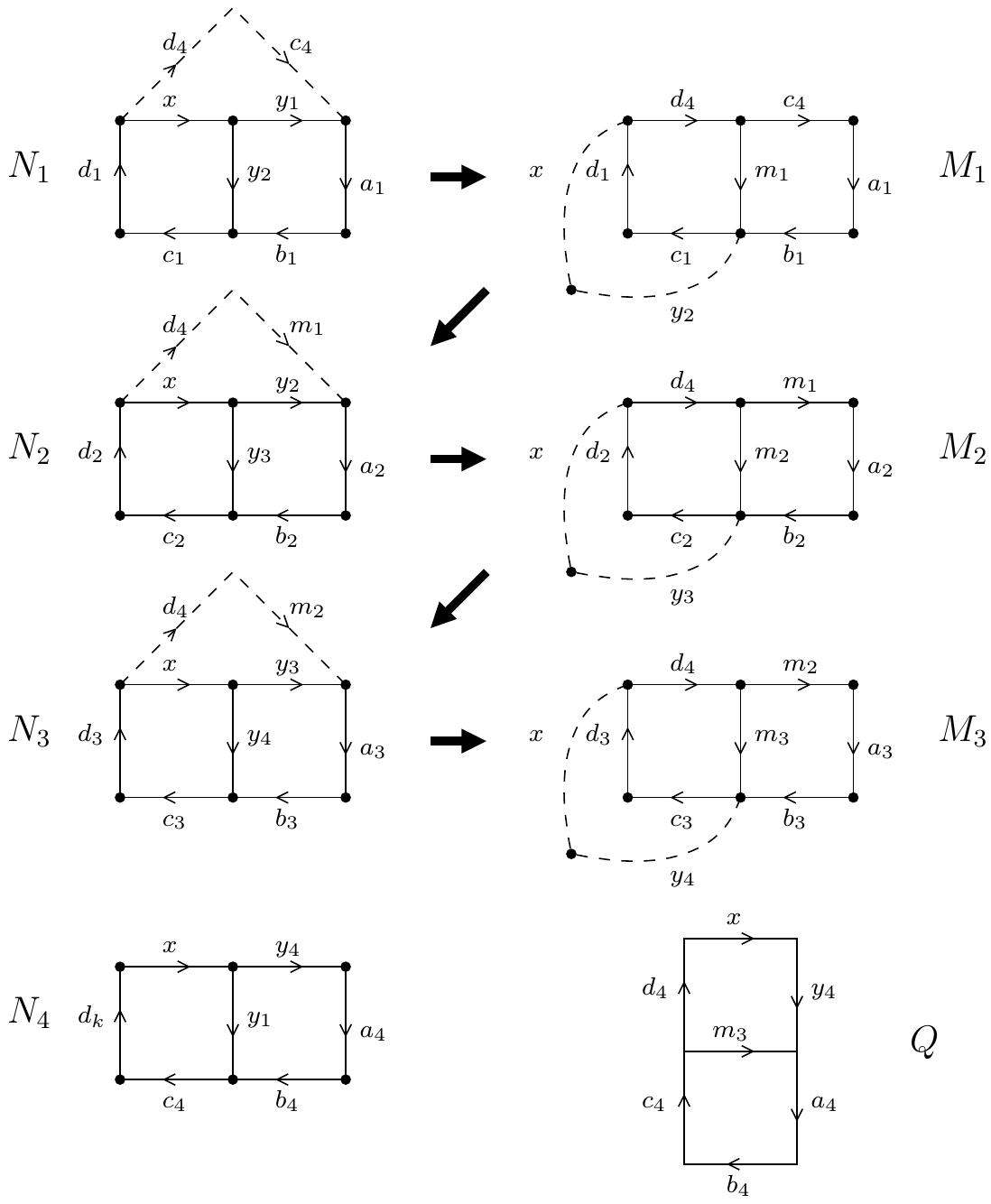}
\caption{No cycles in $\Gamma_x$}\label{fig:noCycles}
\end{figure}

\noindent For $i=1,2,3$, the domino $N_i$ have a region labelled by $a_i b_i y_{i+1}^{-1} y_i$. Hence,  $a_i b_i = y_i^{-1} y_{i+1}$ in $G$ for $i=1,2,3$. The domino $N_4$ has a region labelled by $a_4 b_4 y_1^{-1} y_4$ and so $a_4 b_4 = y_4^{-1} y_1$ in $G$. Thus,
\begin{equation}\label{eq:aibi}
(a_1 b_1) (a_2 b_2) (a_3 b_3) (a_4 b_4) = 1  \quad \text{in $G$}
\end{equation}
and by rearranging we get
\begin{equation}\label{eq:aibi2}
(a_1 b_1) (a_2 b_2) (a_3 b_3) = b_4^{-1} a_4^{-1}  \quad \text{in $G$}
\end{equation}
The domino $M_1$ has a region labelled by $a_1 b_1 m_1^{-1} c_4$ so $a_1 b_1 = c_4^{-1} m_1$ in $G$. For $i=2,3$ the domino $M_i$ has a region labelled by $a_i b_i m_i^{-1} m_{i-1}^{-1}$ and so we have $a_2 b_2 = m_1^{-1} m_2$ and $a_3 b_3 = m_2^{-1} m_3$ in $G$. Combining this with \eqref{eq:aibi2} we get
\begin{equation}\label{eq:ckmk}
b_4^{-1} a_4^{-1} = (c_4 m_1) (m_1^{-1} m_2)  (m_2^{-1} m_3) = c_4 m_3 \quad \text{in $G$}
\end{equation}
Thus, from \eqref{eq:ckmk} there is a relation $m_3 a_4 b_4 c_4$ in $\RR$ (see the bottom region of $Q$). Consequently we can construct the biased domino $(Q,\rho)$ where $\rho$ is labelled by $x y_4 a_4 b_4 c_4 d_4$ but the two biased dominos $(N_4,\mu_4)$ and $(Q,\rho)$ have different types. This shows that $N_4$ is not stable as claimed.
\end{proof}

\section{Barycentric Sub-Division \label{sec:barySubDiv}}

In this section we describe the ``barycentric sub-division'' procedure (BSD procedure, for short) which origins in folklore and is described in \cite{GS91}. The proof of Theorem \ref{thm:mainApplicationC4T4} start by applying this procedure.

Let $\P = \GPres{X|\RR}$ be a presentation of a group $G$ where the set of relations $\RR$ is symmetrically closed and that each generator in $X$ is a piece. The presentation $\P$ and the group $G$ are fixed throughout this section. The result of the BSD procedure is a new presentation $\widetilde{\P} = \GPres{E | \T}$ of $G*F$ where $F$ is a free group of finite rank (the rank of $F$ is the number of relations in $\RR$). We start with the generators $E$. Let $R=x_1 x_2 \cdots x_n$ be a relator in $\RR$. We define $n$ new symbols: $e_1^R,e_2^R,\ldots,e_n^R$ which will be part of the generating set. We say that a generator $e_i^R$ \emph{comes} from the relator $R$. The set $E$ of generators is the following set:
\[
E = \Set{e_i^R | R \in \RR, 1 \leq i \leq |R|}
\]
Next, we describe the set $\T$ of relations. Let $P$ be a piece of the presentation $\P$ and let $R_1 = x_1 x_2 \cdots x_n$ and $R_2 = y_1 y_2 \cdots y_m$ be two relators in $\RR$. Assume that $R_1'$ and $R_2'$ are cyclic permutation of $R_1$ and $R_2$, respectively, such that $P$ is prefix of both $R_1'$ and $R_2'$. Assume also that $R_1' \neq R_2'$. Let $R_1' = PU$. In this decomposition, the prefix $P$ of $R_1'$ starts with the generator $x_i$ and ends with the generator $x_{j-1}$ (note that $j$ may be less than $i$). Similarly, if $R_2' = PV$ then the prefix $P$ of $R_2'$ starts with the generator $y_k$ and ends with the generator $y_{\ell-1}$. When all the above holds, then $\T$ consists of the following relator:
\[
\left(e_i^{R_1}\right)^{-1} e_j^{R_1} \left(e_\ell^{R_2}\right)^{-1} e_k^{R_2}
\]
We will say that the above relator \emph{corresponds} to the piece $P$. We will also say that the sub-words $(e_i^{R_1})^{-1} e_j^{R_1}$ and $(e_\ell^{R_2})^{-1} e_k^{R_2}$ corresponds to the piece $P$.

\begin{observation} \label{obs:letDetPieAndRelator}
The relator $(e_i^{R_1})^{-1} e_j^{R_1} (e_\ell^{R_2})^{-1} e_k^{R_2}$ is determined by the following data:
\[
R_1,\, R_2,\, i,\, j,\, k,\, \ell
\]
where: $R_1$, $R_2$ in $\RR$; $1 \leq i,j \leq |R_1|$; $1 \leq k,\ell \leq |R_2|$; $i \neq j$; $k \neq \ell$.
\end{observation}

We will need the following lemma which follows from Observation \ref{obs:letDetPieAndRelator}.

\begin{lemma} \label{lem:letDetPieAndRelator}
The pieces of the presentation $\widetilde{\P}$ have length at most two.
\end{lemma}
\begin{proof}
Suppose $W=(e_i^{R_1})^{-1} e_j^{R_1} (e_\ell^{R_2})^{-1} e_k^{R_2}$ is a relator in $\T$. Let $U$ be prefix of a cyclic conjugate $V$ of $W$ of length \emph{three}. $U$ must either contain $(e_i^{R_1})^{-1} e_j^{R_1}$ or $(e_\ell^{R_2})^{-1} e_k^{R_2}$. In the second case $U$ must also contain one of the letters $e_i^{R_1}$ or $e_j^{R_1}$ (or their inverses). In the first case $U$ must also contain one of the letters $e_k^{R_2}$ or $e_\ell^{R_2}$ (or their inverses). Thus, from $U$ we can recover the corresponding piece $P$, the two relators $R_1$ and $R_2$, and three of the four indices $i$, $j$, $k$, and $\ell$. Now, from this information we can recover \emph{all} four indices $i$, $j$, $k$, and $\ell$. Thus, it follows that if $U$ is a prefix of a cyclic conjugate $V'$ of an element of $\T$ then $V = V'$. Consequently, there are no pieces of length three.
\end{proof}

In the rest of this section we describe the general properties of the new presentation $\widetilde{\P}$ and its properties when $\P$ satisfies the algebraic $C(4) \& T(4)$ conditions. These properties were already observed in \cite{GS91}. The following is clear.

\begin{observation} \label{obs:lenFourCycRed}
All relations in $\T$ are of length four and cyclically-red\-uced.
\end{observation}

The set $\Set{e_1^R | R \in R}$ has size $|\RR|$ and it generate a free group of rank $|\RR|$. Thus, the presentation 
\[
\GPres{X \cup \Set{e_1^R | R \in R} | \RR}
\]
is a presentation of $G*F$ where $F$ is a free group of rank $|\RR|$. It follows from the above construction that by starting from this presentation one can get the presentation $\widetilde{\P} = \GPres{E | \T}$ by Tietze transformations \cite{LS77}. This implies the following lemma.
\begin{lemma} \label{lem:tietzeTrans}
The presentation $\widetilde{\P}$ presents the group $G * F$ where $F$ is a free group of rank $|\RR|$.
\end{lemma}

To simplify the arguments below, let us adopt the following notation and definition.

\begin{definition}[Rewrite function]
The presentation $\widetilde{\P}$ is a presentation of $G * F$ so both the set $E$ and the set $\Gamma = X \cup \Set{e_1^R | R \in R}$ are generating sets for $G * F$. We define the \emph{rewrite function} $\phi:\ws{E}\to\ws{\Gamma}$. Let
\[
Y = \Set{(e_{i}^R)^{-1}e_{j}^R | R \in \RR \text{ and } 1 \leq i,j \leq |R|}
\]
We first describe an auxiliary function $\psi$ from $Y$ to $\pmX$. Let $(e_{i}^R)^{-1}e_{j}^R$ be an element of $Y$ for some $R \in \RR$ and $1 \leq i,j \leq |R|$. Suppose $R\in\RR$ can be written as $x_1 x_2 \cdots x_n$ and $R'$ is a cyclic conjugate of $R$ which start with $x_i$. Let $W$ be the prefix of $R'$ which ends with $x_{j-1}$. Then,
\[
\psi\left((e_{i}^R)^{-1}e_{j}^R\right) = W
\]
The function $\phi$ takes a word in $\ws{E}$ and produces a word in $\ws{\Gamma}$. We define $\phi$ recursively. If $W \in \ws{E}$ is a word of length one or zero then $\phi(W)=W$. Suppose $W \in \ws{E}$ is of length at least two. Write $W = ab W'$. Then
\[
\phi(W) = \left\{ \begin{array}{ll}
	\psi(ab) \,\phi(W') & \text{if } ab \in Y \\
	a \, \phi(b W') & \text{otherwise}
\end{array}
\right.
\]
In words, the function $\phi$ scans the word $W$ from left to right and replaces each occurrence of an element of $Y$ with a corresponding element of $\ws{W}$.
\end{definition}

It follows from the definition that if $\phi((e_{i}^R)^{-1}e_{j}^R) = W \in \ws{X}$ then $W$ is a sub-word of a cyclic conjugate of $R \in \RR$ (or inverse of $R$). Also, if $R'$ is a cyclic conjugate of $R=x_1 x_2 \cdots x_n \in \RR$ and $R' = W_1 W_2 W_3$ where $W_1$ starts with $x_i$, $W_2$ starts with $x_j$ and $W_3$ starts with $x_k$ then
\[
\phi((e_{i}^R)^{-1}e_{k}^R) = \phi((e_{i}^R)^{-1}e_{j}^R) \, \phi((e_{j}^R)^{-1}e_{k}^R)
\]
Finally, since $\phi$ sends each relator in $\T$ to a relator in $\RR$ (this is a matter of a routine check) we get that 
\[
\overline{\phi}(\overline{W}) = \overline{\phi(W)}
\]
is a homomorphism from $G*F$ to itself. Thus, if $W \in \ws{E}$ is equal to $1$ in $G*F$ then also $\phi(W) = 1$ in $G*F$.

\begin{notation}[Sides and middle segment] \label{not:sidesMidSeg}
If $D$ is a region in a diagram $M$ over $\widetilde{\P}$ then its boundary is labelled by
\[
\left(e_i^{R_1}\right)^{-1} e_j^{R_1} \left(e_\ell^{R_2}\right)^{-1} e_k^{R_2}
\]
for some $R_1,R_2 \in \RR$ and some indices $i$, $j$, $k$, and $\ell$. See Figure \ref{fig:sidesMidSeg}. Assume that $(e_i^{R_1})^{-1} e_j^{R_1}$ corresponds to the piece $P \in \ws{X}$. The boundary of $D$ decomposes as $\partial D = \mu\sigma^{-1}$ where $\mu$ is labelled by $(e_i^{R_1})^{-1} e_j^{R_1}$ and $\sigma$ is labelled by $(e_k^{R_2})^{-1} e_\ell^{R_2}$. We shall call $\mu$ and $\sigma$ the \emph{sides} of $D$ and we shall say that the region $D$ corresponds to the piece $P$. We add to $D$ an auxiliary segment, which is not originally part of the diagram, as follows. The auxiliary segment, denoted here by $\rho$, is a segment labelled by $P$ going from $i(\mu)$ to $t(\mu)$ \emph{inside} the interior of $D$. We call $\rho$ the \emph{middle segment} of $D$. Note once again that the middle segment is an auxiliary construction which we add to the diagram $M$. The initial and terminal vertices of $\mu$ (and also of $\sigma$) will be called vertices of Type A and the other two vertices in $\partial D$ will be called vertices of Type B.

\begin{figure}[ht]
\centering
\includegraphics[totalheight=0.17\textheight]{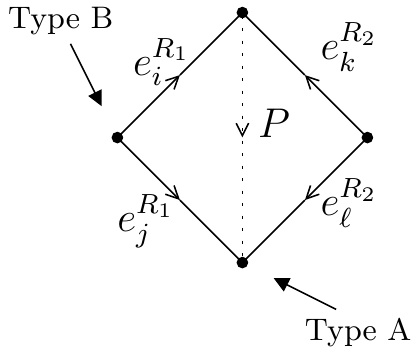}
\caption{Different parts of a region in the new presentation $\widetilde{\P}$}\label{fig:sidesMidSeg}
\end{figure}

\end{notation}

Suppose that $D_1$ and $D_2$ are two neighboring regions in a diagram $M$ over the presentation $\widetilde{\P}$ and that $v$ is a vertex in $\partial D_1 \cap \partial D_2$. It follows routinely from the structure of the relators in $\widetilde{\P}$ that if $v$ is Type A in the boundary of $D_1$ then it is also of Type A in the boundary of $D_2$ and vice versa. Thus, the notion of a type for all vertices of a diagram over $\widetilde{\P}$ is well defined. One may alternatively view a type A vertex as a sink vertex and view a type B vertex as a source vertex. Thus, every vertex in a diagram over  $\widetilde{\P}$ is either a source or a sink. Namely, the set of inner vertices of a diagram over $\widetilde{\P}$ split into two disjoint sets consisting of Type A/sink vertices and Type B/source vertices.

\begin{remark}\label{rem:threeLetImplNotRed}
We remark that by the construction of the relations in $\widetilde{\P}$, if two regions $D_1$ and $D_2$ are neighbors in a diagram $M$ and $\partial D_1 \cap \partial D_2$ contains three edges then $M$ is not reduced since, as shown in Lemma \ref{lem:letDetPieAndRelator}, three letters determine a relator of $\widetilde{\P}$ up to cyclic conjugation.
\end{remark}

Next, we show that in a minimal diagram over $\widetilde{\P}$ each inner region has four neighbors. Also, if $\P$ satisfies the algebraic $C(4) \& T(4)$ conditions and $M$ is a minimal diagram over the presentation $\widetilde{\P}$ then $M$ is a $C(4) \& T(4)$ diagram where edges are labelled by a generator.

\begin{lemma} \label{lem:singGenOrNotMin}
Let $M$ be a van Kampen diagram over $\widetilde{\P}$ with two regions and with connected interior. Then, one of the following holds:
\begin{enumerate}
	\item The inner path of $M$ is labelled by a generator.
	\item $M$ is not minimal.
\end{enumerate}

\end{lemma}
\begin{proof}
\begin{figure}[ht]
\centering
\includegraphics[totalheight=0.13\textheight]{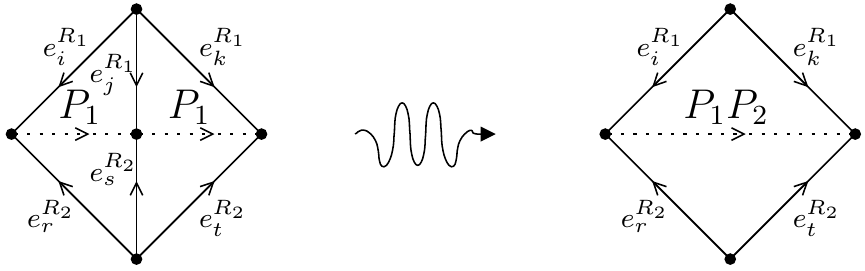}
\caption{Type A inner vertex of valence two}\label{fig:innVertValTwoOfTypeA}
\end{figure}

Suppose $D_1$ and $D_2$ are the two regions of $M$. Let $\mu\sigma^{-1}$ be the boundary cycle of $M$ and $\rho$ the inner path such that $\mu \rho^{-1}$ is a boundary cycle of $D_1$ and $\sigma \rho^{-1}$ is a boundary cycle of $D_2$. Assume that $\rho$ is labelled by $W$. We can assume that the diagram has a boundary of positive length and so we have that $|W| \in \Set{1,2,3}$. If $|W|=1$ then the first case of the lemma holds. If $|W|=3$ then it follows from  Remark \ref{rem:threeLetImplNotRed} that $M$ is not reduced and thus it is not minimal. So, let's assume that $|W|=2$ and that the diagram is reduced. In this case $M$ has one inner vertex $v$ of valance two which is a middle vertex of $\rho$. There are two cases to consider, the first is when $v$ is a Type A vertex and the second is when $v$ is a Type B vertex. These cases are illustrated in Figure \ref{fig:innVertValTwoOfTypeA} for Type A and in Figure \ref{fig:innVertValTwoOfTypeB} for Type B (the labels of the middle segments are noted). Having these figures in mind, it is routine to check that in all cases one may form a diagram with single region, as illustrated in the figures, such that its boundary label is the same as the boundary label of $M$ (this is also done in \cite{GS91}). Consequently, $M$ is not minimal.

\begin{figure}[ht]
\centering
\includegraphics[totalheight=0.13\textheight]{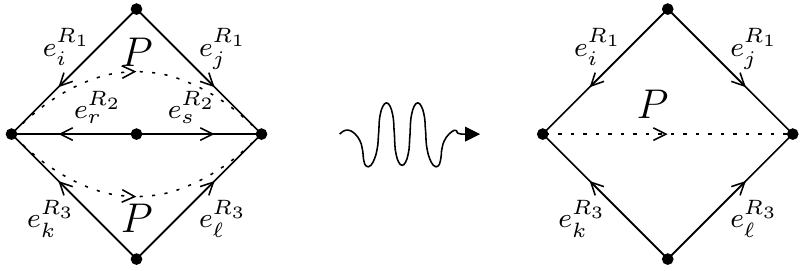}
\caption{Type B inner vertex of valence two}\label{fig:innVertValTwoOfTypeB}
\end{figure}
\end{proof}

\begin{lemma} \label{lem:C4T4NoInnerVertOfValThree}
Suppose that $\P$ satisfies the algebraic $C(4) \& T(4)$ conditions. Let $M$ be a van Kampen diagram over $\widetilde{\P}$ which contain an inner vertex $v$ of valance three. Then, $M$ is not minimal.
\end{lemma}

\begin{figure}[ht]
\centering
\includegraphics[totalheight=0.20\textheight]{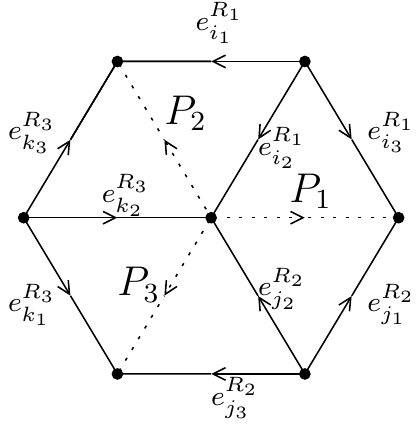}
\caption{Type A inner vertex of valence three}\label{fig:innVertValThreeOfTypeA}
\end{figure}

\begin{proof}
If $M$ is not reduced the it is not minimal. Thus, we will assume that $M$ is reduced. Let $v$ be an inner vertex of valance three and let $D_1$, $D_2$, and $D_3$ be the three regions which contain $v$ in their boundaries. Let $N$ be the sub-map of $M$ which contain the three regions. Using Lemma \ref{lem:singGenOrNotMin} we can assume that the inner edges of $N$ are labelled by a generator (because otherwise by the lemma the sub-diagram $N$ is not minimal and thus the diagram $M$ is not minimal). There are two cases to consider depending on $v$ being a vertex of Type A (Figure \ref{fig:innVertValThreeOfTypeA}) or a vertex of Type B (Figure \ref{fig:innVertValThreeOfTypeB}). 

We start from the case that $v$ is a Type A vertex. Assume that the three regions are labelled by the relators
\begin{description}
	\item[$\quad$] $(R_{j_2}^{R_2})^{-1} R_{j_1}^{R_2} (R_{i_3}^{R_1})^{-1} R_{i_2}^{R_1}$,
	\item[$\quad$] $(R_{i_2}^{R_1})^{-1} R_{i_1}^{R_1} (R_{k_3}^{R_3})^{-1} R_{k_2}^{R_3}$, and
	\item[$\quad$] $(R_{k_2}^{R_3})^{-1} R_{k_1}^{R_3} (R_{j_3}^{R_2})^{-1} R_{j_2}^{R_2}$
\end{description}
which correspond to the pieces $P_1$, $P_2$, and $P_3$, respectively. We also assume that the above pieces are the labels of the middle segments (for convenience we regard these segments as directed outward from $v$). We get that
\begin{description}
	\item[$\quad$] $\phi\left((e_{i_3}^{R_1})^{-1} e_{i_1}^{R_1}\right) = P_1^{-1} P_2\quad$ or $\quad\phi\left((e_{i_1}^{R_1})^{-1} e_{i_3}^{R_1}\right) = P_2^{-1} P_1$;
	\item[$\quad$] $\phi\left((e_{k_3}^{R_3})^{-1} e_{k_1}^{R_3}\right) = P_2^{-1} P_3\quad$ or $\quad\phi\left((e_{k_1}^{R_3})^{-1} e_{k_3}^{R_3}\right) = P_3^{-1} P_2$;
	\item[$\quad$] $\phi\left((e_{j_3}^{R_2})^{-1} e_{j_1}^{R_2}\right) = P_3^{-1} P_1\quad$ or $\quad\phi\left((e_{j_1}^{R_2})^{-1} e_{j_3}^{R_2}\right) = P_1^{-1} P_3$
\end{description}
Thus, $P_1^{-1} P_2$, $P_2^{-1} P_3$, and $P_3^{-1} P_1$ are sub-words of a cyclic conjugates of $R_1$, $R_2$, and $R_2$, respectively, or their inverses. Since $\RR$ is symmetrically closed, we have relators in $\RR$ of the form $P_1 W_1 P_2^{-1}$, $P_2 W_2 P_3^{-1}$, and $P_3 W_3 P_1^{-1}$. This violate the $T(4)$ condition. Consequently, there are no inner vertices of valance three in $M$ which are Type A vertices. 

\begin{figure}[ht]
\centering
\includegraphics[totalheight=0.20\textheight]{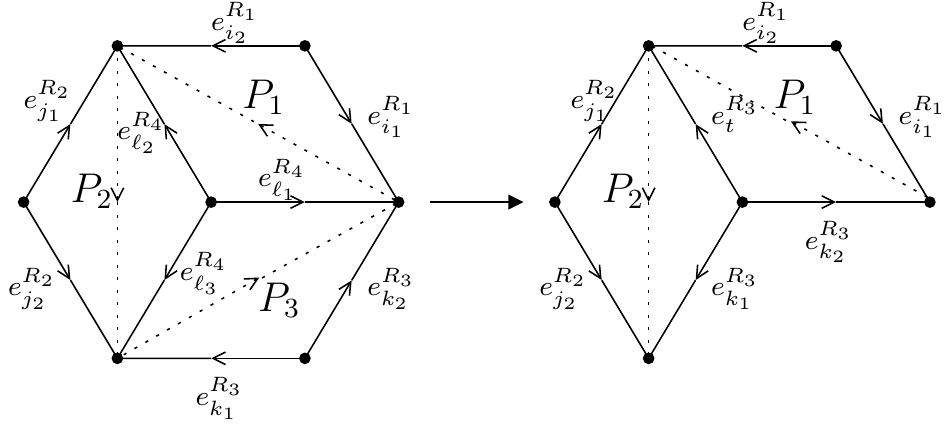}
\caption{Type B inner vertex of valence three}\label{fig:innVertValThreeOfTypeB}
\end{figure}

Suppose now that $v$ is a Type B vertex. Assume that the three regions are labelled by the relators
\begin{description}
	\item[$\quad$] $(R_{i_1}^{R_1})^{-1} R_{i_2}^{R_1} (R_{\ell_2}^{R_4})^{-1} R_{\ell_1}^{R_4}$,
	\item[$\quad$] $(R_{j_1}^{R_2})^{-1} R_{j_2}^{R_2} (R_{\ell_3}^{R_4})^{-1} R_{\ell_2}^{R_4}$, and
	\item[$\quad$] $(R_{k_1}^{R_3})^{-1} R_{k_2}^{R_3} (R_{\ell_1}^{R_4})^{-1} R_{\ell_3}^{R_4}$
\end{description}
which correspond to the pieces $P_1$, $P_2$, and $P_3$, respectively. We also assume that the above pieces are the labels of the middle segments and that these segments form a closed cycle of length three (for convenience, we regard them as directed counter-clockwise around $v$). Finally, we assume that the edges emanating from $v$ are labelled by $R_{\ell_1}^{R_4}$, $R_{\ell_2}^{R_4}$, and $R_{\ell_3}^{R_4}$ (in counter-clockwise order). First, notice that since $(R_{\ell_1}^{R_4})^{-1}R_{\ell_2}^{R_4} (R_{\ell_2}^{R_4})^{-1}R_{\ell_3}^{R_4} (R_{\ell_3}^{R_4})^{-1}R_{\ell_1}^{R_4}$ is freely equal to $1$ we get that 
\[
P_1 P_2 P_3 = \phi\left((R_{\ell_1}^{R_4})^{-1}R_{\ell_2}^{R_4} (R_{\ell_2}^{R_4})^{-1}R_{\ell_3}^{R_4} (R_{\ell_3}^{R_4})^{-1}R_{\ell_1}^{R_4}\right) = 1 \quad \text{in $G$}
\]
Now, it is a well known fact (e.g., by Greendlinger's Lemma) that in a presentation which satisfies the algebraic $C(4)\& T(4)$  condition a freely-reduced product of less than four pieces is never equal to $1$ (since there is no van Kampen diagram with boundary cycle of length three). Thus, by possibly re-ordering the indices, we get that $P_1 P_2 = P_3^{-1}$. Hence, $\phi((E_{k_2}^{R_3})^{-1}E_{k_1}^{R_3}) = P_1 P_2$. Let $t$ be an index such that $\phi((E_{k_2}^{R_3})^{-1}E_{t}^{R_3}) = P_1$ and $\phi((E_{t}^{R_3})^{-1}E_{k_1}^{R_3}) = P_2$. Then, we can replace the diagram $N$ with a new diagram $N'$ which contain two regions labelled by
\begin{description}
	\item[$\quad$] $(R_{i_1}^{R_1})^{-1} R_{i_2}^{R_1} (R_{t}^{R_3})^{-1} R_{k_2}^{R_3}$ and
	\item[$\quad$] $(R_{j_1}^{R_2})^{-1} R_{j_2}^{R_2} (R_{k_1}^{R_3})^{-1} R_{t}^{R_3}$
\end{description}
and which has the same boundary label as $N$; see the right side of Figure \ref{fig:innVertValThreeOfTypeB}. This shows that $N$ is not minimal and thus the diagram $M$ is not minimal.
\end{proof}

\begin{remark} \label{rem:formingInnerVertexOfValance3}
If follows from the proof of Lemma \ref{lem:C4T4NoInnerVertOfValThree} that it is possible to construct a diagram which contains a vertex of valance three (see Figure \ref{fig:innVertValThreeOfTypeB}). This happen if there is a piece $P \in \ws{X}$ of the presentation $\P$ with $|P|\geq2$ and such that there are at least \emph{four} different relators in $\RR$ that contain $P$ as a prefix. In this case, the presentation $\widetilde{\P}$ \emph{does not} satisfy the algebraic $C(4)\& T(4)$ conditions (under which, in \emph{every} van Kampen diagram there are no vertices of valence three). As a concrete example one may consider the following presentation which satisfies the algebraic $C(4)\& T(4)$ condition:
\[
\GPres{a,b|ababab^{-1}a^{-1}b^{-1}a^{-1}b^{-1} = 1}
\]
The symmetric closure of $\Set{ababab^{-1}a^{-1}b^{-1}a^{-1}b^{-1}}$ contains $20$ elements where $ab$ is a prefix of four of them.
\end{remark}

Next theorem summarize the properties of the presentation $\widetilde{\P}$.

\begin{theorem} \label{thm:propOfPTilda}
Let $\P = \GPres{X|\RR}$ be a presentation of a group $G$ which satisfies the algebraic $C(4)\& T(4)$ conditions. Suppose that the set of relations $\RR$ is symmetrically closed and that each generator in $X$ is a piece. Let $\widetilde{\P}$ be the presentation we get by applying the BSD procedure to the presentation $\P$. Then, the following hold:
\begin{enumerate}
  \item $\widetilde{\P}$ presents the group $G * F$ where $F$ is a free group of rank $|\RR|$.
	\item All relators of $\widetilde{\P}$ are of length four and cyclically reduced.
	\item If $M$ is a minimal van Kampen diagram over the presentation $\widetilde{\P}$ then:
	\begin{enumerate}
		\item $M$ is a $C(4) \& T(4)$ diagram.
		\item If $e$ is an inner edge of $M$ then $e$ is labelled by a generator.
	\end{enumerate}
\end{enumerate}
In other words, the presentation $\widetilde{\P}$ is a presentation of the group $G * F$ for which the condition of Theorem \ref{thm:mainThmC4T4} hold.
\end{theorem}
\begin{proof}
Property 1 follows from Lemma \ref{lem:tietzeTrans}. Property 2 follows from Observation \ref{obs:lenFourCycRed}. Property 3(a) follows from Lemma \ref{lem:C4T4NoInnerVertOfValThree} (the $T(4)$ condition) and Lemma \ref{lem:singGenOrNotMin} (the $C(4)$ condition which hold since each edge is labelled by a generator and each relation is of length four so each inner region has four neighbors). Finally, Property 3(b) follows also from Lemma \ref{lem:singGenOrNotMin}.
\end{proof}

\section{Algebraic Condition Imply Bi-Automaticity \label{sec:application}}

In this short section we prove Theorem \ref{thm:mainApplicationC4T4}. Let $G$ be a group which has a presentation $\P$ for which the algebraic $C(4) \& T(4)$ condition hold and the set of relations is symmetrically closed (this may always be assumed). The first step is to apply the BSD procedure. 

We note first that one may assume that each generator is a piece since otherwise we can write $G$ as $G = G' * F$ where:
\begin{enumerate}[(a)]
	\item $G'$ has a presentation where the algebraic $C(4) \& T(4)$ conditions hold and every generator is a piece; and,
	\item $F$ is a free group of finite rank.
\end{enumerate}
The free product $H_1*H_2$ is bi-automatic if and only if both $H_1$ and $H_2$ are bi-automatic \cite[Cor. 4.6]{GS91-2}. Thus, since a free group is bi-automatic, to show that $G$ is bi-automatic it is enough to show that a $G'$ is bi-automatic. Consequently, we can assume that in the presentation $\P$ each generator is a piece.

Let $\widetilde{\P}$ the presentation resulting from applying the BSD procedure to the presentation $\P$. By Theorem \ref{thm:propOfPTilda} the conditions of Theorem \ref{thm:mainThmC4T4} hold and consequently the group presented by $\widetilde{\P}$ is bi-automatic. Now, the group presented by $\widetilde{\P}$ is the group $G*F$ where $F$ is a free group of finite rank. As above, since $G*F$ is bi-automatic also $G$ is bi-automatic and the proof is completed.

\section{Appendix}

In this appendix we discuss the problems in the automaticity proof of groups having a presentation $\GPres{X|\RR}$ which satisfies the algebraic $C(4) \& T(4)$ small cancellation conditions due to S. Gersten and H. Short \cite{GS91}. The proof spans over pg 650--655 of \cite{GS91}. It starts by recalling some basic definition (pg 650). Then, the idea of barycentric subdivision is introduced (top of pg 651) and it is noticed that the construction produces a presentation of the free product $G*F$ where $F$ is a free group (Lemma 5.1).

The construction of the barycentric subdivision is a bit different from the construction we describe in Section \ref{sec:barySubDiv}. We outline the differences. In their construction a new presentation is constructed in which the relations have the form $(e_i^R)^{-1}e_j^R = P$ where $P \in \ws{X}$ is a piece (the generators of the presentation are the original generators plus the new generators of the form $e_i^R$). To handle the new relations the idea of ``admissible diagrams'' is introduced (the first definition at pg 652). There is one to one correspondence between admissible diagrams and the diagrams we describe in Section \ref{sec:barySubDiv}. If one starts from a diagram of the presentation described in Section \ref{sec:barySubDiv} and add all the middle segments (Notation \ref{not:sidesMidSeg}) the result is an admissible diagram over the presentation in \cite{GS91}. On the other hand, is one starts from an admissible diagram and remove all the paths labelled by pieces in $\ws{X}$ the result is a diagram over the presentation described in Section \ref{sec:barySubDiv}. Thus, considering only admissible diagrams essentially means that they concentrate on the same diagrams which we describe in Section \ref{sec:barySubDiv}.

Right after the idea of admissible diagram, the definition of ``allowed word'' is introduced (second definition in pg 652). The lemma that follows (Lemma 5.2) states that allowed words which are equal to $1$ in the group can be read on the boundary cycle of an admissible diagram which its interior vertices have non-positive curvature. The proof of the lemma is in fact a proof that in a minimal admissible diagrams all inner vertices have non-positive curvature. Specifically, the `allowed word' assumption is never used (this is stated at the last paragraph of the proof). Notice that although the proof refers to ``reduced diagrams'' (as assumed in the second paragraph of pg 652) it actually assumes that the diagrams are \emph{minimal}. We also give a proof of non-positive curvature of inner vertices at Section \ref{sec:barySubDiv}. Our arguments are more combinatorial and they refer to a slightly different presentation (as described in previous paragraph) but nevertheless they are essentially the same.

The bulk of the proof is outlined at the end of pg 655. The automatic structure consists of ``allowed representatives''. These are elements which are \emph{geodesics} and for which the so-called ``$B_2$ turning angle condition'' holds (see the first paragraph after the end of the proof at pg 655). It is written that every element of the group has an allowed representative where the proof is ``As in the $B_2$ case'' which is done in Proposition 2.1 (pg 645). The problem in this statement can most clearly seen from the last paragraph of pg 645:
\begin{quote}
\emph{``Notice that if the subword $a_{i-1}a_i$ occurs in another reduced disc diagram, then the turning angle between $a_{i-1}$ and $a_i$ cannot be $\pi/2$. If this were the case, the the non-positive curvature condition would be violated.''}
\end{quote}
The statement is indeed correct under the conditions assumed in Proposition 2.1. It is also crucial to the proof if Proposition 2.1 (because, if a word is included in one diagram and is fixed to make sure the ``$B_2$ turning angle condition'' is maintained then the fix hold for \emph{any} diagram that include the word). However, the same statement it is no longer true in the context of the proof for algebraic $C(4) \& T(4)$ presentations. In Remark \ref{rem:formingInnerVertexOfValance3} we discuss how one may form an inner vertex of valence three. The example there shows that it is possible to construct two minimal diagrams $M_1$ and $M_2$, which their boundary label contains a word $a_{i-1}a_i$ where:
\begin{enumerate}[i)]
	\item in $M_1$ the angle along $a_{i-1}a_i$ is $\pi/2$;
	\item in $M_2$ the angle along $a_{i-1}a_i$ is zero.
\end{enumerate}
(angles as defined in \cite{GS90,GS91}). Consequently, the proof that all elements of the group have an allowed representative is not complete. To be more clear, the proof of the $B_2$ case, which is given in Proposition 2.1, uses the fact that a boundary vertex of valence three cannot become an inner vertex by gluing some region over it. However, in the proof given in pg 655 a boundary vertex of valence three can indeed become an inner vertex by gluing some region over it. In other words, the the definition of ``allowed representatives'' uses a condition which should hold for \emph{every} admissible diagram which includes the representative by referring to the angles along the representative. It follows that vertices of valence three along a possible representative in one diagram which includes it may become vertices of valence two when the representative is included in a different diagram. Thus, one cannot be sure that ``allowed representatives'' actually exist as the ``$B_2$ turning angle condition'' may never holds.

We would like to point out another small problem. It is indicated that, since there are only finitely many configurations which are used to define the allowed representatives, the language of allowed representatives is regular (forth paragraph of pg 655). However, the assumption that allowed representatives are geodesics is problematic in this regard since one cannot readily verify that a given word is a geodesic. Nevertheless, it is not clear how the assumption that all allowed representatives are geodesic is used (it is possible that it may be dropped).


\end{document}